\documentclass{article}
\usepackage{algorithm}
\usepackage{algorithmic}
\usepackage{float}
\usepackage{graphicx}
\usepackage[ngerman,english]{babel}
\usepackage{authblk}
\usepackage{geometry}
\usepackage{amsthm}
\usepackage{amsfonts}
\usepackage{epstopdf}
\usepackage{mathtools}
\usepackage{tikz}
\usepackage[normalem]{ulem}
\usepackage{csquotes}
\usetikzlibrary{calc}
\ifpdf
\DeclareGraphicsExtensions{.eps,.pdf,.png,.jpg}
\else
\DeclareGraphicsExtensions{.eps}
\fi

\newtheorem{theorem}{Theorem}[section]
\newtheorem{lemma}[theorem]{Lemma}
\newtheorem{definition}[theorem]{Definition}
\newtheorem{proposition}[theorem]{Proposition}
\theoremstyle{remark}
\newtheorem{remark}[theorem]{Remark}

\newcommand{\var}{s}
\newcommand{\prior}{\nu_{0}}
\newcommand{\posterior}{\nu^{y}}

\newcommand{\posteriorh}{\nu_{h}^{y}}
\newcommand{\posteriorhat}{\widehat{\nu}^{y}}
\newcommand{\q}[1]{``#1''}

\DeclareMathOperator*{\argmax}{arg\,max}

\title{Bayesian parameter identification in impedance boundary conditions for Helmholtz problems}
\author[1]{Nick Wulbusch}
\author[2]{Reinhild Roden}
\author[2,3]{Matthias Blau}
\author[1]{Alexey Chernov}
\affil[1]{Institut für Mathematik,  Carl-von-Ossietzky Universität, Oldenburg, Germany}
\affil[2]{Institut für Hörtechnik und Audiologie, Jade-Hochschule, Oldenburg, Germany}
\affil[3]{Cluster of Excellence "Hearing4All"}
\affil[]{\textit{nick.wulbusch@uni-oldenburg.de}}
\date{}              
\begin{document}
	\maketitle
	\begin{abstract}
	We consider the problem of identifying the acoustic impedance of a wall surface from noisy pressure measurements in a closed room using a Bayesian approach. The room acoustics is modeled by the interior Helmholtz equation with impedance boundary conditions. The aim is to compute moments of the acoustic impedance to estimate a suitable density function of the impedance coefficient. For the computation of moments we use ratio estimators and Monte-Carlo sampling.
	We consider two different experimental scenarios. In the first scenario, the noisy measurements correspond to a wall modeled by impedance boundary conditions. In this case, the Bayesian algorithm uses a model that is (up to the noise) consistent with the measurements and our algorithm is able to identify acoustic impedance with high accuracy. In the second scenario, the noisy measurements come from a coupled acoustic-structural problem, modeling a wall made of glass, whereas the Bayesian algorithm still uses a model with impedance boundary conditions. In this case, the parameter identification model is inconsistent with the measurements and therefore is not capable to represent them well. Nonetheless, for particular frequency bands the Bayesian algorithm identifies estimates with high likelihood. Outside these frequency bands the algorithm fails. We discuss the results of both examples and possible reasons for the failure of the latter case for particular frequency values.
\end{abstract}
	
	\section{Introduction}

The acoustic properties of rooms (e.g. recording studios, lecture or concert halls) critically depend on the geometry of the room and acoustic properties of the materials the room boundaries are made of. For example, the size and shape of windows made of glass significantly impact acoustic properties of the room. In order to facilitate virtual room design and to ensure reliable simulations the precise identification of parameters, such as the acoustic impedance of the walls and other constituent materials is crucial. These characteristics are typically measured using standardized methods, e.g. an impedance tube \cite{iso_impedanceTube1_1996,iso_impedanceTube2_1998} or a reverberation chamber~\cite{iso_reverberationRoom_2003}.

Frequently it is not possible or not practical to examine isolated pieces of the material. In such a case, one may attempt to identify material parameters by measurements taken within the room itself. Moreover, the actual behavior of the material in the room might be significantly different from measurements done in isolation due to various coupling effects. For these reasons, methods to determine the acoustic impedance from measurements in the room are attractive.

In general, we can classify approaches for parameter identification in inverse problems into two types: classical deterministic and Bayesian approaches. The classical deterministic optimization-based approach involves utilizing an error function and employing methods such as least-squares, potentially with regularization, to solve for the parameter values. The Bayesian approach treats the parameters as random variables and naturally incorporates errors arising from one or more sources and deals with the corresponding uncertainties. 

Previous research in the classical context investigated the deterministic problem to identify the acoustic impedance from measurements in the room. In \cite{piechowicz2013determination} the authors solved the deterministic inverse problem with boundary element method and Nelder-Mead optimization algorithm. This lead to complex values for the acoustic impedance that best fitted the measured data. In \cite{Dutilleux.2002} an evolution strategy approach was used for the optimization of the impedance values and the absorption coefficient. Noise was also considered and good results were achieved if the Signal-to-Noise ratio (SNR) was large enough. In \cite{ANDERSSOHN.2007} the authors considered a related model, where the source of sound is coming from the vibrations of the structure, thus is included in the boundary condition instead of modeling it as a sound source in the interior. They use a genetic global optimization algorithm in combination with Newton's method, which lead to accurate estimations of the acoustic impedance. They however did not consider any noise.

In practice, measurements are typically perturbed by noise. The information about the noise is naturally built into the Bayesian approach, where the parameters are interpreted as random quantities. The goal of the computation is to approximately compute their statistics. The Bayesian approach \cite{stuart2010inverse, Dashti.2017} has several advantages compared to the classical approach. Firstly, the prior distribution serves as a form of regularization with clear interpretation. In contrast, regularization in the deterministic approach often is somewhat arbitrary \cite{stuart2010inverse}. Secondly, the Bayesian inference incorporates inevitable uncertainties directly and transparently \cite{Rappel.2020}. Furthermore, in the deterministic approach, the resulting parameter values are presented with respect to a shared residual value. In contrast, the Bayesian approach yields individual parameter estimates, accompanied by corresponding uncertainty descriptions for each parameter separately \cite{Rappel.2020}.

In this paper we consider the Bayesian approach to estimate the acoustic impedance on specific boundary regions in a room. For this, instead of computing a single (complex) value for the acoustic impedance we interpret the acoustic impedance as a random variable. We compute its statistical moments via ratio estimators and Monte-Carlo sampling to characterize the impedance. We consider the measurement setup of one sound source and several microphones at different positions, which represents the typical measurement procedure. The room acoustic is modeled using the interior Helmholtz equation, i.e., the time-harmonic wave equation, with impedance boundary conditions that represents locally reacting wall impedances. To solve the resulting partial differential equation, we employ the finite element method. We prove well-posedness of our problem and show convergence of the mean squared error for functionals of the acoustic impedance, which includes statistical moments of its real and imaginary part, using Monte Carlo sampling and ratio estimators. In the numerical experiments we consider two different scenarios. In the first scenario we utilize synthetic data generated using the same parametric model that is used for solving the inverse problem, i.e., impedance boundary conditions on the specific boundary regions. In this case the data are consistent with the parametric model and one would expect that the Bayesian algorithm will be able to recover correct parameter values. For this case we demonstrate convergence of the Bayesian algorithm numerically. In the second scenario the synthetic data are generated using a coupled acoustic-structural model where the boundary region is modeled as a glass wall, which is not locally reacting. Although in general the locally reacting impedance model does not describe the physics of the glass wall well, for some frequency intervals the Bayesian algorithm finds values of acoustic impedance having high likelihood. Outside these frequency ranges, however, the results did not represent the measurement well due to the model-data misfit. We investigate this somewhat surprising behaviour and explain it by the local structure of the eigenmodes and the distribution of the resonances.

We also point out that the Bayesian parameter identification for acoustic problems has been also applied in recent works \cite{JonasM.Schmid.,engel2019application}. The short note \cite{JonasM.Schmid.} contains a numerical study of a two-dimensional acoustic problem with a different type of excitation.
In \cite{engel2019application} the authors consider an acoustical problem of source detection. Here, the uncertainty lies in the source of the Helmholtz equation, i.e., the position, number and amplitude of the sources. The authors adopt a Bayesian framework, mathematically analyzed the problem and proved convergence of the sequential Monte Carlo method, a method to sample directly from the posterior. In our work we are considering the uncertainty in the boundary condition instead of the source and utilize ratio estimators to compute moments of the posterior instead of sequential Monte Carlo.

This paper is structured as follows. In section \ref{sec:preliminaries} we discuss the model Helmholtz equation and review some basic definitions, theorems regarding existence and regularity of the solution and error estimates. In  section \ref{sec:discretization} we discuss the discretization of the problem in a finite element setting and collect convergence estimates required in that follows. In section \ref{sec:bayesian} we introduce the Bayesian setting and analyze the convergence of the proposed algorithm. Section \ref{sec:numerics} shows the numerical experiments where the theoretical error bounds are confirmed if the data are consistent with the model. We also demonstrate the the behavior in the case of model-data misfit. Section \ref{sec:outlook} contains the outlook.

\section{Preliminaries}
\label{sec:preliminaries}
We consider an interior Helmholtz problem on a bounded convex polygonal or polyhedral domain $D\subset \mathbb{R}^{d}$, $d=2,3$. The boundary $\Gamma \coloneqq \partial D$ is decomposed into two disjoint open sets, $\Gamma_{N}$ and $\Gamma_{R}$, such that $\Gamma = \overline{\Gamma_{N}} \cup \overline{\Gamma_{R}}$. The complex-valued acoustic pressure is then modeled as solution of the problem
\begin{equation}
	\label{eq:model_problem}
	\left\{
	\begin{split}
		-\Delta p -k^2 p \quad=\quad f  &\quad\text{ in } D,\\
		\frac{\partial p}{\partial n} + \frac{i\omega \rho}{Z} p \quad=\quad 0 &\quad\text{ on } \Gamma_{R},\\
		\frac{\partial p}{\partial n} \quad=\quad 0 &\quad\text{ on } \Gamma_{N},
	\end{split}
	\right.
\end{equation}
with angular frequency $\omega$, density $\rho$, acoustic impedance $Z$ and wave number $k=\frac{\omega}{c}$, where $c$ is the speed of sound. Here, we assume that $Z$ is bounded with uniformly positive real part and piecewise constant on a fixed partition $\cup_{i=1}^{n} \overline{\Gamma_{R}^{(i)}} = \Gamma_{R}$, that is
\begin{equation}\label{eq:U}
	Z\in U \coloneqq \{Z \in L^{\infty}(\Gamma_{R}) ~\vert~ Z \equiv z_{i} \in \mathbb{C} \text{ on } \Gamma_{R}^{(i)} \text{ with } Re(z_{i}) \ge C > 0\}.
\end{equation}
Note that $Re(Z)\ne 0$ on some portion of the boundary is a sufficient condition for the well-posedness of \eqref{eq:model_problem}. However, in the practically relevant case of a passive surface absorbing energy $Re(Z)$ is strictly positive.


\subsection{Weak formulation}
The \textit{weak formulation} is obtained by multiplying the equation \eqref{eq:model_problem} by the complex conjugate of a test function $q\in H^{1}(D)$ and integrating over $D$. The weak formulation reads: Find $p\in H^{1}(D)$ such that for all $q \in H^{1}(D)$
\begin{equation}
	\label{eq:WeakFormulation}
	\int_{D} \nabla p \cdot \nabla \overline{q} \,dx +\int_{\Gamma_{R}} \frac{i\omega \rho}{Z} p \overline{q} \,ds - k^2 \int_{D} p\overline{q}\,dx = \int_{D} f \overline{q}\,dx.
\end{equation}
We define the sesquilinear form $a:H^{1}(D)\times H^{1}(D)\mapsto \mathbb{C}$ and antilinear form $\ell: H^{1}(D) \rightarrow \mathbb{C}$ as follows
\begin{align}
	\label{eq:SesquilinearForm}
	a(p,q) &\coloneqq \int_{D} \nabla p \cdot \nabla \overline{q} \,dx + \int_{\Gamma_{R}} \frac{i\omega \rho}{Z}  p \overline q\,ds - k^{2} \int_{D} p\overline{q}\,dx,\\
	\ell(q) &\coloneqq \int_{D} f \overline{q} \,dx.
\end{align}
The weak formulation \eqref{eq:WeakFormulation} is well-posed for $f\in L^{2}(D)$, see, e.g., \cite{chandler2012numerical}. However, in the experiments, a point source is assumed as the sound source. Therefore, the right-hand side $f$ is modeled as a Dirac delta distribution $\delta^{\var}$ located at a point $\var$ in the interior of $D$. The theory for the weak formulation does not guarantee well-posedness in this case, hence the so-called \q{very weak formulation} is used instead. Additionally, we define source and measurement domain that are distinct from each other, such that source and microphone positions are not too close to each other or to the boundary.

\begin{definition}[Source and measurement domain, cf. {\cite[Definition 2.5]{engel2019application}}]
	For $\kappa>0$ the source domain $D_{\kappa}\subset D$ is a set satisfying $\mathrm{dist}(D_{\kappa},\Gamma)>\kappa$. The measurement domain is defined as
	\begin{equation*}
		M_{\kappa} \coloneqq \{x \in D : \mathrm{dist}(x,D_{\kappa})>\kappa \text{ and } \mathrm{dist}(x,\Gamma)>\kappa\}.
	\end{equation*}	
\end{definition}

\subsection{Very weak formulation for point source excitations}
Let $s\in D_{\kappa}$ be the location of the point source. The \textit{very weak formulation} of \eqref{eq:WeakFormulation} with Dirac delta right hand side is given by: Find $G_{Z}^{\var} \in L^{2}(D)$ such that
\begin{equation}
	\label{eq:VeryWeakFormulation}
	\int_{D} G_{Z}^{\var} (-\Delta \overline{q})dx - k^{2}\int_{D} G_{Z}^{\var} \overline{q} dx = \overline{q}(\var),
\end{equation}
for all $q\in\{q\in H^{2}(D): \frac{\partial q}{\partial n} + \frac{i\omega \rho}{Z} q = 0 \text{ on } \Gamma_{R} \text{ and } \frac{\partial q}{\partial n} = 0 \text{ on } \Gamma_{N}\}$. Since the solution $G_{Z}^{s}$ depends on of $s$ and $Z$, we explicitly track this dependency in the notation. The right hand side of the equation comes as a result of the defining property of the Dirac delta distribution, i.e.,
\begin{equation*}
	\int_{D} \delta^{\var}(x) \overline{q}\,dx = \overline{q}(\var).
\end{equation*}

\begin{proposition}[cf. {\cite[Proposition 2.8]{engel2019application}}]
	\label{pro:VeryWeakSolutionExistence}
	The very weak formulation \eqref{eq:VeryWeakFormulation} has a unique solution $G_{Z}^{\var}\in L^{2}(D)$ for any $s \in D_{\kappa}$. Additionally there exists a constant $C_{\kappa}>0$ depending on $Z$ and $\kappa$ but not on $\var$ such that
	\begin{equation}
		\label{eq:Gxbounds}
		\lVert G_{Z}^{\var}\lVert_{H^{2}(M_{\kappa})}, \lVert G_{Z}^{\var} \rVert_{W^{2,\infty}(M_{\kappa})}\le C_{\kappa} \quad \text{ for all } \var \in D_{\kappa}.
	\end{equation}
	For $Z\in U$ the constant $C_{\kappa}$ can be chosen independently of $Z$. Furthermore, there exists a constant $\tilde{C}_{\kappa}>0$, such that for all $z\in M_{\kappa}$ it holds that
	\begin{equation*}
		\lvert G_{Z^{(1)}}^{\var}(z)-G_{Z^{(2)}}^{\var}(z)\rvert \le \tilde{C}_{\kappa} \lVert Z^{(1)} - Z^{(2)} \rVert_{L^{\infty}(\Gamma_{R})}
	\end{equation*}
	for all piecewise constant $Z^{(1)}, Z^{(2)}$ with $Re(Z^{(1)}),Re(Z^{(2)})\ge C >0.$
\end{proposition}
\begin{proof}
	Let $\Phi^{\var}$ be the fundamental solution of the Helmholtz operator, that is
	\begin{equation}
		\label{eq:fundamentalSolution}
		\Phi^{\var}(z) \coloneqq \begin{cases}
			\dfrac{1}{2\pi} Y_{0}\left(k\lVert \var - z \rVert \right),& \text{ if } d = 2,\\[2ex]
			\dfrac{\exp\left(-k\lVert \var-z \rVert\right)}{4\pi \lVert \var - z \rVert},& \text{ if } d = 3,
		\end{cases}
	\end{equation}
	where $Y_{0}$ is the Bessel function of second kind and zero order, see e.g. \cite{dautray2012mathematical}.	Let $p_{Z}^{\var}$ be the unique weak solution of 
	\begin{alignat*}{8}
		-\Delta p_{Z}^{\var}- k^{2} p_{Z}^{\var} &&\quad = &\quad  0 &&\quad\text{ in } D,\\
		\frac{\partial p_{Z}^{\var}}{\partial n} + \frac{i\omega\rho}{Z} p_{Z}^{\var} &&\quad = &\quad -\frac{\partial \Phi^{\var}}{\partial n} - \frac{i\omega\rho}{Z}\Phi^{\var} &&\quad \text{ on } \Gamma_{R},\\
		\frac{\partial p_{Z}^{\var}}{\partial n} &&\quad = &\quad -\frac{\partial \Phi^{\var}}{\partial n} &&\quad\text{ on } \Gamma_{N},
	\end{alignat*}
	i.e., $p_{Z}^{\var}\in H^{1}(D)$ is the unique solution of
	\begin{align*}
		\int_{D} \nabla p_{Z}^{\var} \cdot \nabla \overline{q}\,dx + \int_{\Gamma_{R}}\frac{i\omega\rho}{Z}p_{Z}^{\var}\overline{q}d\Gamma_{R} &- k^{2} \int_{D}p_{Z}^{\var}\overline{q}\,dx\\
		& = -\int_{\Gamma_{R}} \left[ \frac{i\omega\rho}{Z}\Phi^{\var} + \frac{\partial \Phi^{\var}}{\partial n}\right]\overline{q}d\Gamma_{R} - \int_{\Gamma_N} \frac{\partial \Phi^{\var}}{\partial n} \overline{q}d\Gamma_{N}
	\end{align*}
	for all $q\in H^{1}(D)$. Thus, the solution $G_{Z}^{\var}$ of \eqref{eq:VeryWeakFormulation} can be written as
	\begin{equation}
		\label{eq:GxPhip}
		G_{Z}^{\var} = \Phi^{\var} + p_{Z}^{\var}.
	\end{equation}
	In \cite[Lemma 3.1, Theorem 3.3, Chapter 3.2]{bermudez2004finite}  the existence and uniqueness of $p_{Z}^{\var}$ and hence of $G_{Z}^{\var}$ was proven, by showing that $p_{Z}^{\var}$ satisfies G\aa rding's inequality and thus the problem satisfies the Fredholm alternative, i.e., existence is a consequence of uniqueness which was shown in \cite[Lemma 3.1]{bermudez2004finite}. Following \cite{bermudez2004finite}, since $D$ is convex, the $H^{2}$-estimate is given by
	\begin{equation*}
		\lVert p_{Z}^{\var} \rVert_{H^{2}(D)} \le C \left[\left\lVert \frac{\partial \Phi^{\var}}{\partial n} + \frac{i\omega\rho}{Z}\Phi^{\var}\right\rVert_{H^{\frac{1}{2}}(\Gamma_{R})} + \left\lVert \frac{\partial \Phi^{\var}}{\partial n}\right\rVert_{H^{\frac{1}{2}}(\Gamma_{N})}\right],
	\end{equation*}
	where the fundamental solution $\Phi^{\var}$ and its derivatives are bounded uniformly in $x$ for $\lVert x-\var \rVert>\kappa$, see \cite[proof of Theorem 2.8]{engel2019application}. The $W^{2,\infty}$-estimate $\lVert p_{Z}^{\var}\rVert_{W^{2,\infty}}(M_{\kappa})<C_{\kappa}$ was proven in \cite[Lemma 3.4]{bermudez2004finite}. Tracking the constants in the proof of \cite[Lemma 3.4]{bermudez2004finite} $C_{\kappa}$ can be chosen independently of $Z$ if $\frac{1}{Z}$ is bounded, which is the case for $Re(Z)\ge C >0$. 
	It remains to show the Lipschitz continuity in $Z$. The Sobolev emdebbing \cite[Theorem 4.12]{adams2003sobolev} implies that $C(M_{\kappa})\subset H^{2}(M_{\kappa})$ for $d = 2,3$. Thus, for $z\in M_{\kappa}$ 
	\begin{equation*}
		\lvert G_{Z^{(1)}}^{\var}(z) - G_{Z^{(2)}}^{\var}(z)\rvert \le \lVert G_{Z^{(1)}}^{\var} - G_{Z^{(2)}}^{\var}\rVert_{C(M_{\kappa})} \le \tilde{C} \lVert G_{Z^{(1)}}^{\var} - G_{Z^{(2)}}^{\var}\rVert_{H^{2}(M_{\kappa})} 
	\end{equation*}
	Now observe that for $Z^{(1)},Z^{(2)}\in \mathbb{C}$ with $Re(Z^{(1)}),Re(Z^{(2)})\ge C>0$ we have
	\begin{align*}
		\lVert G_{Z^{(1)}}^{\var} - G_{Z^{(2)}}^{\var} \rVert_{H^{2}(M_{\kappa})} &\le \lVert p_{Z^{(1)}}^{\var} - p_{Z^{(2)}}^{\var} \rVert_{H^{2}(D\backslash D_{\kappa})} + \lVert \Phi^{\var} - \Phi^{\var}\rVert_{H^{2}(D\backslash D_{\kappa})}\\
		&\le C_{1}\left[\left\lVert \frac{i\omega \rho}{Z^{(1)}} \Phi^{\var} + \frac{\partial \Phi^{\var}}{\partial n} - \left(\frac{i\omega \rho}{Z^{(2)}}\Phi^{\var}+\frac{\partial \Phi^{\var}}{\partial n}\right) \right\rVert_{H^{1/2}(\Gamma_{R})}\right]\\
		&= C_{1}  \left\lVert i\omega \rho \left( \frac{1}{Z^{(1)}} - \frac{1}{Z^{(2)}}\right)\Phi^{\var} \right\rVert_{H^{1/2}(\Gamma_{R})},
	\end{align*}
	where the right-hand side admits the following upper bound
	\begin{align*}
		C_{1}  \left\lVert i\omega \rho \left( \frac{1}{Z^{(1)}} - \frac{1}{Z^{(2)}}\right)\Phi^{\var} \right\rVert_{H^{1/2}(\Gamma_{R})}&\le C_{2} \lvert \omega \rho \rvert \lVert \Phi^{\var} \rVert_{H^{2}(D\backslash D_{\kappa})} \left\lVert \frac{1}{Z^{(1)}}-\frac{1}{Z^{(2)}} \right\rVert_{L^{\infty(\Gamma_{R})}}\\
		&\le C_{3} \left\lVert \frac{Z^{(2)}-Z^{(1)}}{Z^{(1)}Z^{(2)}}\right\rVert_{L^{\infty}(\Gamma_{R})}\\
		&\le  \tilde{C}_{\kappa} \left\lVert Z^{(1)} - Z^{(2)}\right\rVert_{L^{\infty}(\Gamma_{R})}.
	\end{align*}
\end{proof}

\section{Finite Element Discretization}
\label{sec:discretization}
In our approach, we discretize the problem using a quasi-uniform family of shape-regular triangulations $\{\mathcal{T}_{h}\}_{h>0}$ of $D$. Each element $T\in\mathcal{T}_{h}$ is a triangle or a tetrahedron with diameter $h_{T}$. The diameter of the largest ball contained in $T$ is denoted by $\rho_{T}$. The maximum diameter of all elements in $\mathcal{T}_{h}$ is the mesh size $h\coloneqq \max_{T\in \mathcal{T}_{h}} h_{T}$. We assume there exist constants $c_{1}, c_{2} >0$, such that for all $T \in \mathcal{T}_{h}$ and for all $h>0$
\begin{equation*}
	\frac{h_{T}}{\rho_{T}} \le c_{1}, \quad \frac{h}{h_{T}} \le c_{2}.
\end{equation*}
We then define the finite element space $V_{h}$ on the triangulation $\mathcal{T}_{h}$. The functions in $V_{h}$ are globally continuous on $D$ and piecewise linear in each $T\in\mathcal{T}_{h}$. The discrete weak formulation of our model problem \eqref{eq:model_problem} is: Find $p_{h} \in V_{h}$ such that 
\begin{equation}
	\label{eq:discrete_problem}
	a(p_{h},q_{h}) = \ell(q_{h}), \qquad q_{h} \in V_{h}
\end{equation}
where $a$ is given as in \eqref{eq:SesquilinearForm}. The antilinear functional
\[
\ell(q_{h}) = \overline{q}_{h}(\var)
\]
as in the very weak formulation \eqref{eq:VeryWeakFormulation}. Notice that even for the point source excitation $f$ the weak formulation \eqref{eq:discrete_problem} is well posed on $V_h$, and thus can be used for the numerical approximation of \eqref{eq:VeryWeakFormulation}. The following theorem states pointwise convergence of these approximations.

\begin{theorem}[cf. {\cite[Theorem 2.14]{engel2019application}}]
	\label{thm:DiscretisationErrorG}
	Let $Z\in U$. Let $G_{Z,h}^{\var}\in V_{h}$ be the discrete solution to \eqref{eq:discrete_problem} with $\ell(q) = \overline{q}(\var)$. Then for any $h\in(0,h_{0}]$, $z \in M_{\kappa}$ and $\var \in D_{\kappa}$ there exist $h_{0},C_{\kappa}>0$ such that
	\begin{equation*}
		\lvert G_{Z}^{\var}(z)-G^{\var}_{Z,h}(z)\rvert \le C_{\kappa}\lvert \ln h \rvert h^{2} 
	\end{equation*}
\end{theorem}

\begin{proof}
	See appendix of \cite{engel2019application}. The proof requires the bound $\lVert G_{Z}^{\var} \rVert_{W^{2,\infty}(M_{\kappa})} < C_{\kappa}$ independently of $Z$, which has been shown in Proposition \ref{pro:VeryWeakSolutionExistence}.
\end{proof}

\begin{proposition}
	\label{cor:contGhinZ}
	There exist $h_{0}, C_{\kappa}>0$ such that for all $z\in M_{\kappa}$ and for all $h\in \left(\left.0,h_{0}\right]\right.$ the discrete solution mapping $G_{\cdot,h}^{\var}(z) : U \rightarrow \mathbb{C}$ is Lipschitz continuous in $Z$, that is
	\begin{equation*}
		\lvert G_{Z_{1},h}^{\var}(z) - G_{Z_{2},h}^{\var}(z)\rvert \le C_{\kappa} \lVert Z_{1} - Z_{2}\rVert_{L^{\infty}(\Gamma_{R})}.
	\end{equation*}
\end{proposition}
\begin{proof}
	The proof is similar to the continuous case in Proposition \ref{pro:VeryWeakSolutionExistence} by replacing Green's function with its discrete version.
\end{proof}

\section{Bayesian framework}
\label{sec:bayesian}
Since measurements are typically corrupted by noise, it is reasonable to work with models that take noise into account. In what follows we interpret the impedance $Z$ as random variable and estimate its statistical moments using the Bayesian approach and the ratio estimators. This provides quantitative information on the sensitivity of the parameter with respect to the data, for example if the variance of $Z$ is large, then multiple values and/or (large) regions are candidates for the true parameter value. If, however, the variance is very low, then only a small region in the domain of $Z$ produces accurate results if the underlying model is a good representation for the data. 

\subsection{Continuous posterior moments}
Let $\mathcal{F}:U\rightarrow H^{1}(D)$ and $\mathcal{O}:H^{1}(D)\rightarrow\mathbb{C}^{m}$ denote the forward map and observation operator, respectively, where $m\in\mathbb{N}$. The composition of both operators, denoted by $\mathcal{G}=\mathcal{O}\circ\mathcal{F}:U\rightarrow\mathbb{C}^{m}$, is given by
\begin{equation*}
	\mathcal{G}(Z) = \big(G_{Z}^{\var}(x_{1}),  G_{Z}^{\var}(x_{2}), \dots , G_{Z}^{\var}(x_{m})\big)^\top.
\end{equation*}
The problem is to determine an unknown element $Z\in U$ from noisy observations
\begin{equation}
	y = \mathcal{G}(Z) + \eta,
\end{equation}
where $\eta$ is a realization of a multivariate complex normal distributed random variable $\mathcal{CN}(0,\Gamma,C)$ with a Hermitian and non-negative covariance matrix $\Gamma$ and symmetric relation matrix $C$. The corresponding density of such a random variable is proportional to
\begin{equation*}
	\rho_{\eta} \propto \exp\left(-\frac{1}{2}\lVert z \rVert_{\Sigma}^{2}\right),\quad \forall z \in \mathbb{C}^{m}.
\end{equation*}
where $\Sigma \in \mathbb{C}^{2m\times 2m}$ is a positive definite complex matrix given by
\begin{equation*}
	\Sigma = \begin{pmatrix}
		\Gamma & C\\
		\overline{C} & \overline{\Gamma}
	\end{pmatrix} \qquad \text{ and } \qquad
	\lVert z \rVert_{\Sigma}^{2} \coloneqq \left(\overline{z}^{\top}, z^{\top}\right) \Sigma^{-1} \begin{pmatrix}
		z\\
		\overline{z}
	\end{pmatrix}.
\end{equation*} We assume $Z$ to be distributed according to a prior measure $\prior$ on $(U,\mathcal{B})$, where $\mathcal{B}$ is the Borel $\sigma$-algebra. Bayes' Theorem gives the following relation for the Radon-Nikodym derivative of the posterior measure $\posterior$ with respect to the prior measure $\prior$ (see, e.g., \cite[Theorem 6.31]{stuart2010inverse}):
\begin{equation}
	\label{eq:RadonNikodym}
	\frac{d\posterior}{d\prior}(Z) = \frac{\theta(Z,y)}{\Lambda(y)},
\end{equation}
where
\begin{align}
	\theta(Z,y) &= \exp(-\Psi(Z,y)),\\
	\Psi(Z,y) &=  \frac{1}{2}\lVert y - \mathcal{G}(Z)\rVert_{\Sigma}^2 \label{eq:Potential}\\
	\Lambda(y) &= \mathbb{E}_{\prior}[\theta(Z,y)]. \label{def-Lambda}
\end{align}
Here, $\theta$ is called the \emph{likelihood}, $\Psi$ the potential, $\log \theta = -\Psi$ the \emph{log-likelihood} and $\Lambda$~is a normalization constant. The expectation $\mathbb{E}_{\nu}[\phi]$ of a function $\phi: U\rightarrow \mathbb{R}$ with respect to a measure $\nu$ on $(U,\mathcal{B})$ is defined as
\begin{equation}\label{def-ratioEst}
	\mathbb{E}_{\nu}[\phi] = \int_{U} \phi(Z) d\nu(Z).
\end{equation} 

\begin{remark}
	Since the potential $\Psi(Z,y)$ is nonnegative, the likelihood $\theta(Z,y)$ can only take values between 0 and 1 and the value 1 is only attained if the measurement data $y$ exactly match the value $\mathcal{G}(Z)$. This is only possible if \emph{(i)} the parametric model exactly match the data $y$ and \emph{(ii)} the measurements contain no noise. The strength of the noise and the number of observation points clearly impact what can be treated as \enquote{a high value of the likelihood}. In particular, for the experimental setting of section \ref{sec:3Dimp} we will see that log-likelihood values of $\log \theta(Z,y) \sim -10$ (and larger) can be seen as high.
\end{remark}

Let $\phi:U\rightarrow \mathbb{R}$ be a $\prior$-measureable functional. If  the posterior measure $\posterior$ is well-defined (see Theorem \ref{thm:posteriorWellDefined} below), then, according to \eqref{eq:RadonNikodym}, the expected value of $\phi(Z)$ under the posterior measure can be expressed as
\begin{equation}
	\label{eq:ExpectedValue}
	\mathbb{E}_{\posterior}[\phi(Z)] = \frac{Q(y)}{\Lambda(y)},
\end{equation}
where the normalization constant $\Lambda(y)$ has been defined in \eqref{def-Lambda} and
\begin{equation}
	Q(y) = \mathbb{E}_{\prior}[\theta(Z,y)\phi(Z)]
\end{equation}
Notice that the posterior expectation $\mathbb{E}_{\posterior}[\phi(Z)]$ is a ratio of two expectations with respect to the prior measure. A computable version of \eqref{eq:ExpectedValue}, see \eqref{eq:ratioEstimator} below, is called a \emph{ratio estimator}.
Later we will chose $\phi(Z) = Re(Z)^{k}$ or $\phi(Z) = Im(Z)^{k}$ to compute $k$-th moments of the real or imaginary part of $Z$, respectively. From now on let $\prior$ be a measure on $(U,\mathcal{B})$.

The aim of the remainder of this subsection is to show that the posterior measure is well-defined.
\begin{definition}
	\label{def:BochnerSpace}
	Let $\nu$ be a measure on $(U,\mathcal{B})$ and $k\ge 1$. We say $\phi \in L_{\nu}^{k}(U)$ if
	\begin{equation}
		\lVert \phi \rVert_{L_{\nu}^{k}(U)} = \begin{cases}
			\left( \int_{U} \lvert \phi(Z) \rvert^{k} \,d\nu\right)^{\frac{1}{k}}, & 1 \le k \le \infty\\
			\mathrm{esssup}_{Z\in U} \lvert \phi(Z) \rvert, & k=\infty
		\end{cases}
	\end{equation}
	is finite.
\end{definition}

\begin{proposition}
	\label{prop:ObservationOperatorMeasureable}
	The operator $\mathcal{G}$ is $\prior$-measurable and bounded, i.e.,
	\[
	\lVert \mathcal{G}(Z) \rVert_{\Sigma} \le K.
	\]	
\end{proposition}
\begin{proof}
	The measurability and the boundedness follow from the continuity and boundedness of $G_{Z}^{\var}(z)$ in $Z$ for all measurement positions $z\in M_{\kappa}$ from Proposition \ref{pro:VeryWeakSolutionExistence}. Note that $K$ depends on $Z$ or more precisely on $\frac{1}{Z}$. But $\frac{1}{Z}$ is bounded since we assume $Z\in U$, hence $K$ can be chosen independently of $Z$.
\end{proof}

\begin{lemma}[{\cite[Lemma 3.6]{engel2019application}}]
	\label{lem:PropertiesOfPotential}
	The potential $\Psi$ and the prior measure $\prior$ satisfy
	\begin{enumerate}
		\item[(i)] There exists a constant $K>0$ such that
		\[
		0 \leq \Psi(Z,y)\le K + \lVert y \rVert^{2}_{\Sigma}, \quad \text{for all } Z\in U, y\in \mathbb{C}^{m}
		\]
		\item[(ii)] For every $y\in\mathbb{C}^{m}$ the map $\Psi(\cdot,y):U\rightarrow \mathbb{R}$ is $\prior$-measurable.
		\item[(iii)] For every $\varrho>0$ there exists a constant $C_{\varrho}>0$ such that 
		\[
		\lvert \Psi(Z,y_{1})-\Psi(Z,y_{2})\rvert \le C_{\varrho}\lVert y_{1}-y_{2}\rVert_{\Sigma}
		\]
		for all $Z\in U$ and for all $y_{1}, y_{2}\in B_{\varrho}(0) \coloneqq \{y \in \mathbb{R}^{m} ~:~ \lVert y \rVert_{2} < \varrho \}$.
	\end{enumerate}
\end{lemma}
\begin{proof}
	\begin{enumerate}
		\item[(i)] The lower bound follows from the definition of the potential \eqref{eq:Potential}, whereas the upper bound is a consequence the triangle inequality and Proposition \ref{prop:ObservationOperatorMeasureable} since
		\begin{align*}
			\Psi(Z,y) = \frac{1}{2} \lVert y - \mathcal{G}(Z) \rVert_{\Sigma}^{2}\le \lVert y \rVert_{\Sigma}^{2} + \lVert \mathcal{G}(Z) \rVert_{\Sigma}^{2} \le \lVert y \rVert_{\Sigma}^{2} + K.
		\end{align*}
		\item[(ii)] Follows from the measurability of $\mathcal{G}$ and the continuity of $\Psi(Z,\cdot)$.
		\item[(iii)] We have
		%
		\begin{align*}
			\Psi(Z,y_{1})  - \Psi(Z,y_{2}) 
			&= \frac{1}{2}  \|  y_{1}-\mathcal{G}(Z) \|_{\Sigma}^2 - \frac{1}{2} \|y_{2}-\mathcal{G}(Z)\|_{\Sigma}^2  \\
			&= \frac{1}{2}   (y_{1}-\mathcal{G}(Z),y_{1}-y_{2})_{\Sigma} + \frac{1}{2} (y_{1}-y_{2},y_{2}-\mathcal{G}(Z))_{\Sigma}  
		\end{align*}
		Now, by Cauchy-Schwarz and the triangle inequality we continue as
		\begin{align*}
			|\Psi(Z,y_{1})  - \Psi(Z,y_{2})| &\le \frac{1}{2}\lVert y_{1} - y_{2} \rVert_{\Sigma} \left(\lVert y_{1}-\mathcal{G}(Z) \rVert_{\Sigma} + \lVert y_{2} - \mathcal{G}(Z) \rVert_{\Sigma}\right)\\
			&\le \lVert y_{1}-y_{2}\rVert_{\Sigma} \left(\lVert \mathcal{G}(Z) \rVert_{\Sigma} + \frac{1}{2}\lVert y_{1} \rVert_{\Sigma} + \frac{1}{2}\lVert y_{2} \rVert_{\Sigma}\right)\\
			&\le C_{\varrho}\lVert y_{1} -y _{2} \rVert_{\Sigma},
		\end{align*}
		where the last inequality holds since $\mathcal{G}(Z)$ is bounded and $y_{1},y_{2}\in B_{\varrho}(0)$.
	\end{enumerate}	
\end{proof}
To show the well-posedness of the posterior $\posterior$ it is essential to show the following bounds on the normalization constant $\Lambda(y)$.
\begin{lemma}
	\label{lem:NormalizationConstant}
	For the normalization constant $\Lambda(y)$ there exist $\alpha>0$ such that
	\begin{align*}
		&\alpha \exp\left(-\lVert y \rVert_{\Sigma}\right) \le \Lambda(y) \le 1 \qquad \text{ for all } y \in \mathbb{C}^{m}.
	\end{align*}
\end{lemma}
\begin{proof}
	The upper bound follows directly from the definition \eqref{def-Lambda} of $\Lambda(y)$
	\begin{align*}
		\Lambda &= \mathbb{E}_{\prior}\left[ \theta(Z,y) \right] = \int\limits_{U} \theta(Z,y)\,d\prior(Z) = \int\limits_{U} \exp(-\Psi(Z,y))\,d\prior(Z)\le 1,	
	\end{align*}
	since $\prior$ is a probability measure on $U$ and $\Psi$ is nonnegative. For the lower bound we use property (ii) from Lemma \ref{lem:PropertiesOfPotential} to get
	\begin{align*}
		\Lambda = \int\limits_{U} \exp(-\Psi(Z,y))\,d\prior(Z) \ge \int\limits_{U} \exp(-(K+\lVert y \rVert_{\Sigma}))\,d\prior(Z) = \alpha \exp(-\lVert y \rVert_\Sigma)
	\end{align*}
	with $\alpha \coloneqq \exp(-K)$.
\end{proof}

We can now state the well-posedness of the posterior $\posterior$ and the existence of moments in the following theorem, cf. \cite[Theorem 15]{Dashti.2017}.
\begin{theorem} 
	\label{thm:posteriorWellDefined}
	Let $y\in \mathbb{C}^{m}$ be a fixed set of measurements. The posterior defined in \eqref{eq:RadonNikodym} is well defined. Furthermore, if $\phi \in L_{\nu_{0}}^{p}(U)$ then $\phi \in L_{\nu^{y}}^{p}(U)$ for all $1 \leq p \leq \infty$ and a fixed $y$.
\end{theorem}
\begin{proof}
	We show that the Radon-Nikodym derivative \eqref{eq:RadonNikodym} is bounded for all $Z\in U$. This follows since $\theta(Z,y) \leq 1$ and $\Lambda(y)$ is uniformly positive for $Z \in U$, cf. Lemma \ref{lem:NormalizationConstant}. 
	To show that $\phi \in L_{\prior}^{p}(U)$ implies $\phi \in L_{\posterior}^{p}(U)$ we observe by \eqref{def-ratioEst} that
	\begin{align*}
		\lVert \phi \rVert_{L_{\posterior}^{p}(U)}^p = \mathbb{E}_{\posterior}[\lvert \phi \rvert^{p}]   =   \frac{1}{\Lambda(y)} \mathbb{E}_{\prior}\left[ \lvert \phi \rvert^{p} \theta(\cdot,y)\right]
		\leq \frac{1}{\Lambda(y)} \lVert \phi \rVert_{L_{\prior}^{p}(U)}^p    < \infty,
	\end{align*}
	since $\theta(Z,y) \leq 1$ and $\Lambda(y)$ is uniformly positive for $Z \in U$ by  Lemma \ref{lem:NormalizationConstant}.
\end{proof}

Note that the expected value of $\phi$ under the posterior measure $\posterior$ is stable with respect to the data, i.e.,
\[
\lVert \mathbb{E}_{\mu^{y_{1}}}[\phi]-\mathbb{E}_{\mu^{y_{2}}}[\phi]\rVert_{X} \le c\lVert y_{1}-y_{2}\rVert_{\Sigma}.
\]
This can be shown by the means of the Hellinger distance, using the properties of the potential proven in Lemma \ref{lem:PropertiesOfPotential}. See \cite[Theorem 16]{Dashti.2017} and \cite[Lemma 6.37]{stuart2010inverse} for the details.

\subsection{Computable posterior moments}
In practice, the weak formulation is discretized and solved approximately. A discretization naturally introduces perturbation of the forward map $\mathcal{F}: U \rightarrow V$. To extend the Bayesian framework to this setting, we introduce the discrete forward map $\mathcal{F}_h: U \rightarrow V_h$, which now depends on the discretization parameter $h$. For a Finite Element discretization from section \ref{sec:discretization} $h$ stands for the mesh size. We can then define the discrete version of the operator $\mathcal{G}$ introduced in section \ref{sec:bayesian} as $\mathcal{G}_h \coloneqq \mathcal{O} \circ \mathcal{F}_h: U \rightarrow \mathbb{C}^m$. To adapt the likelihood, potential, and normalization constant to the discrete setting, we define 
\begin{align}
	\theta_{h}(Z,y) &= \exp(-\Psi_{h}(Z,y)),\label{eq:discreteLikelihood} 
	& & 
	Q_{h}(y) = \mathbb{E}_{\prior}[\theta_{h}(Z,y)\phi(Z)],\\
	\Psi_{h}(Z,y)  &=  \frac{1}{2}\lVert y - \mathcal{G}_{h}(Z)\rVert_{\Sigma}^2,  &&
	\Lambda_{h}(y) = \mathbb{E}_{\prior}[\theta_{h}(Z,y)].\label{eq:discretePotential} 
\end{align}
\begin{remark}
	\label{rem:PropertiesDiscretization}
	The discrete potential $\Psi_{h}$ satisfies the same properties as $\Psi$ given in Lemma \ref{lem:PropertiesOfPotential} in the continuous case, where the constants and sets are independent of $h$ \cite[Theorem 4.3]{engel2019application}. Furthermore, Lemma \ref{lem:NormalizationConstant} also holds for the discrete normalization constant $\Lambda_{h}$.
\end{remark}

If the Finite Element approximation converges, one expects that the size of perturbations $\mathcal{G} - \mathcal{G}_h$ and $\Psi - \Psi_h$ become smaller as $h \to 0$. The following two lemmas quantify this statement in a precise way.

\begin{lemma}[{\cite[Lemma 4.1]{engel2019application}}]
	\label{lem:DiscretizationErrorObservationOperator}
	There exist $h_{0},C>0$ such that for every $h\in(0,h_{0}]$ and every $Z\in U$ the discrete observation operator $\mathcal{G}_{h}$ satisfies
	\begin{align}
		\label{eq:discreteObservationOperator1}
		\lVert \mathcal{G}_{h}(Z)\rVert_{\Sigma} &\le C \qquad \text{ and }\\
		\label{eq:discreteObservationOperator2}
		\lVert \mathcal{G}(Z)-\mathcal{G}_{h}(Z)\rVert_{\Sigma} &\le C \lvert \ln h \rvert h^{2}.
	\end{align}
	Furthermore, $\mathcal{G}_{h}$ is $\prior$-measurable.
\end{lemma}
\begin{proof}
From Theorem \ref{thm:DiscretisationErrorG} we have 
\begin{equation*}
\lvert G_{Z}^{\var}(z) - G_{Z,h}^{\var}(z)\rvert \le C_{\kappa}\lvert \ln h \rvert h^{2},  \quad \text{for all } z \in M_{\kappa} 
\end{equation*}
with $C_{\kappa}>0$ independent of $Z$ and $h$. Since $\mathcal{G}$ and $\mathcal{G}_{h}$ are just point evaluations of $G_{Z}^{\var}$ and $G_{Z,h}^{\var}$, \eqref{eq:discreteObservationOperator2} holds (with a different constant that does not depend on $h$). Further we have by the triangle inequality
\begin{equation*}
\lVert \mathcal{G}_{h}(Z)\rVert_{\Sigma} \le \lVert \mathcal{G}_{h}(Z)-\mathcal{G}(Z)\rVert_{\Sigma} + \lVert \mathcal{G}(Z)\rVert_{\Sigma}
\end{equation*}
which proves the boundedness of $\mathcal{G}_{h}$ since $\lVert \mathcal{G}(Z) \rVert_{\Sigma}$ is bounded due to Proposition \ref{prop:ObservationOperatorMeasureable}. The $\prior$-measureability follows from the continuity of $G_{Z,h}^{\var}$ with respect to $Z$ shown in Proposition \ref{cor:contGhinZ}. 
\end{proof}
\begin{lemma}[{\cite[Lemma 4.2]{engel2019application}}]
\label{lem:DiscretisationErrorPotential}
There exist $C,h_{0}>0$ such that for every $h\in(0,h_{0}]$ the discrete potential $\Psi_{h}$ satisfies for all $Z\in U$, $y\in\mathbb{C}^{m}$
\[
\lvert \Psi(Z,y)-\Psi_{h}(Z,y)\rvert \le C  \left(1 + \lVert y \rVert_{\Sigma}\right) \lvert \ln h \rvert h^{2}.
\]	
\end{lemma}
\begin{proof}
Acting as in the proof of Lemma Lemma Lemma \ref{lem:PropertiesOfPotential} (iii) we get 
\begin{align*}
\Psi(Z,y) - \Psi_{h}(Z,y)  &= \frac{1}{2}  \lVert y - \mathcal{G}(Z) \rVert_{\Sigma}^{2} - 
\frac{1}{2}\lVert y - \mathcal{G}_{h}(Z)\rVert_{\Sigma}^{2}  \\
&= \frac{1}{2}  (y-\mathcal{G}(Z), \mathcal{G}_{h}(Z)-\mathcal{G}(Z))_{\Sigma} 
+  \frac{1}{2}  (\mathcal{G}_{h}(Z)-\mathcal{G}(Z),y-\mathcal{G}_{h}(Z))_{\Sigma}.
\end{align*}
Thus, by the triangle and the Cauchy-Schwarz inequality we get
\begin{align*}
| \Psi(Z,y) - \Psi_{h}(Z,y)| &\leq  \lVert \mathcal{G}_{h}(Z)- \mathcal{G}(Z) \rVert_{\Sigma} \left(\lVert y \rVert_{\Sigma} + \frac{1}{2}\lVert \mathcal{G}(Z) \rVert_{\Sigma}+ \frac{1}{2}\lVert \mathcal{G}_{h}(Z) \rVert_{\Sigma}\right) \\
&\le C  \left(1 + \lVert y \rVert_{\Sigma} \right) \lvert \ln h \rvert h^{2},
\end{align*}
where we have used Poposition \ref{prop:ObservationOperatorMeasureable} and Lemma \ref{lem:DiscretizationErrorObservationOperator} in the last step.
\end{proof}

\begin{remark}
Lemma \ref{lem:DiscretizationErrorObservationOperator} and Lemma \ref{lem:DiscretisationErrorPotential} imply that (see also \cite[Theorem 17]{Dashti.2017} for a generalization of this): There exists $h_{0}>0$ such that the posterior measure $\posteriorh$ is well-defined for $h\in \left(0,h_{0}\right]$.	
\end{remark}

To make the posterior expectation \eqref{eq:ExpectedValue} fully computable, we replace the prior expectations by a computable approximation. In this paper we use for simplicity the standard Monte Carlo method, i.e., the empirical mean
\begin{equation*}
E_{N}[X] = \frac{1}{N}\sum\limits_{i=1}^{N} X^{i}.
\end{equation*}
for independent and identically distributed samples  $X^{1}, \dots, X^N$ of $X$.
Possible more sophisticated alternatives are, e.g., Quasi Monte Carlo methods  \cite{Sloan.1994,Niederreiter.1992,Kuo.2016} and Multilevel Monte Carlo methods  \cite{Heinrich.2001,Giles.2008,Barth2011MultilevelMC}. The fully computable ratio estimator is now given by 
\begin{equation}
\label{eq:ratioEstimator}
\frac{\widehat{Q}_{h,N}}{\widehat{\Lambda}_{h,N}} := \frac{E_{N}[Q_{h}]}{E_{N}[\Lambda_{h}]}.
\end{equation}
The following theorem states the convergence of the ratio estimator in the mean-square sense with respect to the prior measure.

%
%
%
%
%
\begin{theorem} 
\label{thm:MSE}
Let $\phi \in L_{\prior}^{2}(U)$. Then
\begin{equation*}\label{eq:MSE}
\mathrm{MSE}\left( \frac{\widehat{Q}_{h,N}}{\widehat{\Lambda}_{h,N}} \right) \coloneqq \mathbb{E}_{\prior}\left[ \left( \frac{Q}{\Lambda} - \frac{\widehat{Q}_{h,N}}{\widehat{\Lambda}_{h,N}} \right)^{2} \right] \le C \left( h^{4} \lvert \log h \rvert^{2} + \frac{1}{N} \right).
\end{equation*}
\end{theorem}
\begin{proof}
For the $\mathrm{MSE}$ the following holds
\begin{align*}
\mathrm{MSE}\left( \frac{\widehat{Q}_{h,N}}{\widehat{\Lambda}_{h,N}} \right)  &= \mathbb{E}_{\prior}\left[ \left( \frac{Q}{\Lambda} - \frac{\widehat{Q}_{h,N}}{\widehat{\Lambda}_{h,N}}\right)^{2} \right] = \mathbb{E}_{\prior}\left[ \left( \frac{Q\widehat{\Lambda}_{h,N}-Q\Lambda+Q\Lambda-\widehat{Q}_{h,N}\Lambda}{\Lambda \widehat{\Lambda}_{h,N}} \right)^{2} \right]\\
&\le 2 \mathbb{E}_{\prior}\left[ \left( \frac{Q\left( \widehat{\Lambda}_{h,N}-\Lambda\right)}{\Lambda \widehat{\Lambda}_{h,N}} \right)^{2} \right] + 2 \mathbb{E}_{\prior}\left[ \left( \frac{\left( Q-\widehat{Q}_{h,N}\right)\Lambda}{\Lambda\widehat{\Lambda}_{h,N}} \right)^{2} \right]\\
&= 2\mathbb{E}_{\prior}\left[ \left( \frac{Q}{\Lambda \widehat{\Lambda}_{h,N}} \right)^{2} \left(\widehat{\Lambda}_{h,N} - \Lambda\right)^{2} \right] + 2 \mathbb{E}_{\prior}\left[ \frac{1}{\widehat{\Lambda}_{h,N}^{2}}\left( Q-\widehat{Q}_{h,N}\right)^{2}\right]
\end{align*}
Since $Q=\mathbb{E}_{\prior}[\theta(\cdot,y)\phi] \le \mathbb{E}_{\prior}[\lvert\phi\rvert]$, we know that $Q$ is bounded if $\phi \in L_{\prior}^{1}(U)$. Moreover, Lemma \ref{lem:NormalizationConstant} and remark \ref{rem:PropertiesDiscretization} imply that for a fixed $y\in\mathbb{C}^{m}$ $\Lambda$ and that $\widehat{\Lambda}_{h,N}$ are bounded from below uniformly in $Z\in U$. Thus, the following estimate holds
\begin{equation*}
\mathrm{MSE}\left( \frac{\widehat{Q}_{h,N}}{\widehat{\Lambda}_{h,N}} \right) \le C(y,\kappa) \left( \mathbb{E}_{\prior}\left[ \left( \widehat{\Lambda}_{h,N} - \Lambda\right)^{2} \right] + \mathbb{E}_{\prior}\left[ \left( Q-\widehat{Q}_{h,N} \right)^{2} \right]\right)
\end{equation*}
Here, the last expression is well defined, since $\phi \in L_{\prior}^{2}(U)$. The following Lemma \ref{lem:MSEsplitting} concludes the proof by showing the convergence results for 
\[
\mathbb{E}_{\prior}\left[ \left( Q-\widehat{Q}_{h,N} \right)^{2} \right].
\]
The convergence for the other summand, i.e., the MSE of $\Lambda$, follows since $\Lambda$ corresponds to a special case of $Q$ with $\phi \equiv 1$.
\end{proof}
The following Lemma concludes the proof of Theorem \ref{thm:MSE} by analyzing the MSE for $Q$. This is done by splitting the MSE into bias and variance and showing convergence of these quantities.
\begin{lemma}
\label{lem:MSEsplitting}
For the MSE the following splitting holds 
\begin{equation}\label{MSEsplitting}
\mathrm{MSE}(\widehat{Q}_{h,N}) = \mathbb{E}_{\prior}\left[ (Q-\widehat{Q}_{h,N} )^{2} \right]= \mathrm{Bias}(\widehat{Q}_{h,N})^{2} + \mathrm{\mathbb{V}ar}(\widehat{Q}_{h,N})
\end{equation}
where 
\begin{align}
\mathrm{Bias}(\widehat{Q}_{h,N}) &= |Q - \mathbb{E}_{\prior}[\widehat{Q}_{h,N} ] | \le C_B h^{2} \lvert \log h \rvert, \label{bias-bnd}\\
\mathrm{\mathbb{V}ar}(\widehat{Q}_{h,N}) &= \mathbb{E}_{\prior}\left[ \left(Q_{h} - Q_{h,N}\right)^{2} \right] \le \frac{C_V}{N}\label{var-bnd}
\end{align}
and constants $C_B$ and $C_V$ independent of $h$ and $N$.
\end{lemma}
\begin{proof}
Observe that
\begin{align*}
\mathrm{MSE}(\widehat{Q}_{h,N}) &= \mathbb{E}_{\prior} [ (Q-\widehat{Q}_{h,N} )^{2}  ] = \mathbb{E}_{\prior} [ (Q-Q_{h})^{2}  ] + \mathbb{E}_{\prior} [ (Q_{h} - \widehat{Q}_{h,N})^{2} ]
\end{align*}
since $Q$ and $Q_h$ are deterministic and therefore
\begin{align*}
&\mathbb{E}_{\prior} [(Q-Q_{h})(Q_{h}-\widehat{Q}_{h,N})  ] = (Q-Q_{h})(Q_{h}-\mathbb{E}_{\prior}[\widehat{Q}_{h,N}]) = 0,
\end{align*}
where we used that $\widehat{Q}_{h,N}$ is an unbiased estimator of $Q_h$. 	
This implies \eqref{MSEsplitting} since $Q$ and $Q_h = \mathbb{E}_{\prior}[\widehat{Q}_{h,N} ]$ are deterministic.
The estimate \eqref{bias-bnd} follows from
Lemma \ref{lem:DiscretisationErrorPotential} and the Lipschitz continuity of $\exp(-\cdot)$ on $[0,\infty)$, since
\begin{align*}
\mathrm{Bias}(\widehat{Q}_{h,N}) &= |\mathbb{E}_{\prior}\left[ Q-Q_{h}\right] | 
\leq \mathbb{E}_{\prior}\left[  |\phi(Z)| \;|\theta(Z,y)-\theta_{h}(Z,y)| \right]\\
& \le  \mathbb{E}_{\prior}[|\phi(Z)|] \sup_{Z\in \mathop{\rm supp}(\prior)}\left\lvert\Psi(Z,y)-\Psi_{h}(Z,y)\right\rvert \\
& \le C \;\mathbb{E}_{\prior}[\phi(Z)] \; (1+\lVert y \rVert_{\Sigma})\; h^{2} \lvert \log h \rvert.
\end{align*}
Concerning the variance, we observe
\begin{align*}
\mathrm{\mathbb{V}ar}(\widehat{Q}_{h,N}) &= \mathbb{E}_{\prior}\left[ \left( \mathbb{E}_{\prior}\left[ \phi(Z) \theta_{h}(Z,y) \right] - E_{N}\left[ \phi(Z) \theta_{h}(Z,y) \right] \right)^{2}\right]\\
&= \mathbb{E}_{\prior} \left[ \left( \frac{1}{N} \sum\limits_{i=1}^{N} \bigg(\phi(Z^{i}) \theta_{h}(Z^{i},y) - \mathbb{E}_{\prior}\left[ \phi(Z^{i}) \theta_{h}(Z^{i},y)\right] \bigg)\right)^{2} \right]
\end{align*}
where $Z^{i}$ are independent copies of $Z$. This implies
\begin{align*}
\mathrm{\mathbb{V}ar}(\widehat{Q}_{h,N}) &=\frac{1}{N^2}  
\mathbb{E}_{\prior} \left[ \bigg( \sum_{i=1}^N \phi(Z^i) \theta_{h}(Z^i,y) - \mathbb{E}_{\prior}\left[ \phi(Z^i) \theta_{h}(Z^i,y)\right] \bigg)^{2} \right] 
\\
&= \frac{1}{N} \mathbb{E}_{\prior}\left[ \left(\phi(Z) \theta_{h}(Z,y) - \mathbb{E}_{\prior}\left[ \phi(Z) \theta_{h}(Z,y)\right] \right)^{2} \right] \leq \frac{C_{V}}{N}.
\end{align*}
Since $\theta_{h}(Z)\le 1$ for all $Z \in U$, the last inequality is satisfied with $C_V = \mathbb{E}_{\prior} [ \phi(Z)^2]$. This implies \eqref{var-bnd} and concludes the proof thereby.
\end{proof}

\section{Numerical experiments}
\label{sec:numerics}
In this section, we demonstrate the numerical performance of our approach using various model problems in two and three space dimensions. In all experiments the aim is to determine the complex-valued acoustic impedance parameter $Z = Z_R + iZ_I$ and thereby fit the impedance boundary condition
\begin{equation}\label{IBC}
\frac{\partial p}{\partial n} + \frac{i \omega \rho}{Z} p = 0 \quad \text{ on } \Gamma_R
\end{equation}
at the part of the boundary $\Gamma_R \subset \partial D$. The parameter $Z$ is assumed to be constant (in  section \ref{subsect:numMSE} piecewise constant) on $\Gamma_R$. 
\begin{figure}[htbp]
\centering
\begin{tikzpicture}[scale=1]	
\coordinate (A1) at (0cm,0cm);
\coordinate (A2) at (3cm,0cm);

\coordinate (A3) at (3cm,3.5cm);
\coordinate (A4) at (0cm,3.5cm);

\fill[blue!60, opacity=0.2] (A1) -- (A2) -- (A3) -- (A4) -- cycle;
\draw[thick] (A1) -- (A2) -- (A3) -- (A4) -- (A1);

\draw[thick,arrows=->] (A1) -- (3.5cm,0) node[right] {$$};
\draw[thick,arrows=->] (A1) -- (0,4cm) node[left] {$$};

\node at (3cm,-0.25cm) {$3$};
\node at (-0.3cm,3.5cm) {$3.5$};

\draw[thick,blue] (A1) -- (A4);
\draw[thick,red] (A1) -- (A2);
\node at (-0.3cm,1.75cm) {\textcolor{blue}{$\Gamma_{R}^{(1)}$}};
\node at (1.5cm,-0.25cm) {\textcolor{red}{$\Gamma_{R}^{(2)}$}};
\node at (3.1cm,3.7cm) {$\Gamma_{N}$};

\coordinate (B1) at (0.25cm,0.25cm);
\coordinate (B2) at (2.75cm,0.25cm);
\coordinate (B3) at (2.75cm,3.25cm);
\coordinate (B4) at (0.25cm,3.25cm);

\fill[blue!60, opacity=0.4] (B1) -- (B2) -- (B3) -- (B4) -- cycle;
\draw (B1) -- (B2) -- (B3) -- (B4) -- (B1);

\foreach \x in {3,...,27}
\foreach \y in {3,...,32}
{
	\fill (\x*0.1cm,\y*0.1cm) circle (1pt);
}

\filldraw[fill=white] (1cm,1cm) circle (15pt);
\filldraw[fill=blue!60, opacity=0.2, draw=black] (1cm,1cm) circle (15pt);
\filldraw[fill=red!40!white, draw=black] (1cm,1cm) circle (10pt);
\filldraw[fill=red!40!white, draw=black] (1cm,1cm) circle (1pt);

\draw[black] (1cm,1cm) -- (-0.5cm,1cm);
\node at (-1.3cm,1cm) {\textcolor{black}{source}};
\draw[red] (0.9cm,0.9cm) -- (-0.5cm,0.7cm);
\node at (-0.6cm,0.6cm) {\textcolor{red}{$D_{\kappa}$}};
\draw[blue] (2cm,2cm) -- (3.5cm,3cm);
\node at (3.6cm,3.2cm) {\textcolor{blue}{$M_{\kappa}$}};

\end{tikzpicture}
\caption{Domain for model problem in 2D. Black dots indicate the grid of possible microphone positions.}
\label{fig:2d}
\end{figure}
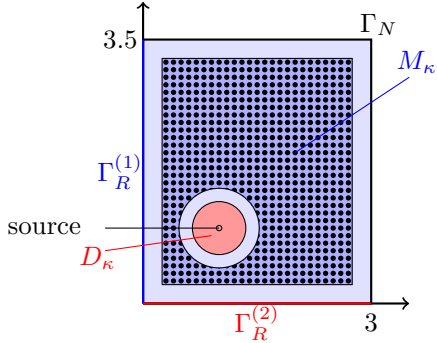
\subsection{The general experimental setup}
\label{ubsect:numSetup}
We consider a convex room $D \subset \mathbb{R}^d$, which is either a rectangle for $d=2$ or a cuboid for $d=3$, with a fixed sound point source at $s$ within the source domain $D_\kappa\subset D$ and $m$ microphones placed at $m$ different locations $x_1,\dots,x_m$ within the measurement domain $M_\kappa \subset D \setminus D_\kappa$. The $m$ measurement positions are randomly chosen in $M_{\kappa}$, based on a regular grid with $L$ points in the room, see figure \ref{fig:2d}.  We draw $m$ samples from a discrete uniform distribution with sample space $\{1,\dots,L\}$, where the indices of points that are too close (i.e., the distance between them is smaller than $\kappa$) to the source or boundary $\partial D$ are excluded.  To ensure consistency across experiments, the sound source location remains the same for a fixed experimental scenario. The computational algorithm is outlined as Algorithm \ref{alg1}.

\begin{algorithm}
\begin{itemize}
\item[(1)] For a given source of acoustic excitation get noisy measurements $y = (y_1,\dots,y_m)$, where $y_j$ is the point value of the acoustic pressure $p(x_j)$ at a point in the measurement domain $x_j \in M_\kappa$.
\item[(2)] Describe the a priori knowledge of the parameter $Z$ by the prior density $\nu_0$.
\item[(3)] Draw $N$ independent samples $Z^1,\dots, Z^N$ of the parameter $Z$ distributed according to the prior density $\prior$.
\item[(4)] For each sample $Z^i$ solve numerically the model problem \eqref{eq:model_problem} with the local impedance boundary condition on $\Gamma_R$. Collect the point values $\mathcal{G}_h(Z^i) = (p_h(x_1),\dots,p_h(x_m))$, where $p_h$ is the deterministic numerical solution of \eqref{eq:model_problem} for $Z = Z^i$.
\item[(5)] Evaluate the likelihood $\theta_h(Z^i,y)$ of the sample $Z^i$ according to \eqref{eq:discreteLikelihood}, \eqref{eq:discretePotential}.
\item[(6)] Evaluate the posterior mean and the variance of $Z_* \in \{Z_R, Z_I\}$ by the ratio estimator \eqref{eq:ratioEstimator}, i.e., $\widehat{\Lambda}_{h,N} = \frac{1}{N} \sum_{i=1}^N \theta_h(Z^i,y)$ and
\begin{equation} \label{approxMoments}
	\begin{aligned}
		\widehat{\mathcal{M}}^1_y[Z_*]_{h,N} &= \frac{1}{\widehat{\Lambda}_{h,N}} \left( \frac{1}{N} \sum_{i=1}^N Z_*^i \theta_h(Z^i_*,y) \right), \\
		\widehat{\mathcal{M}}^2_y[Z_*]_{h,N} &= \frac{1}{\widehat{\Lambda}_{h,N}} \left( \frac{1}{N} \sum_{i=1}^N (Z_*^i - \widehat{\mathcal{M}}^1_y[Z_*]_{h,N})^2 \theta_h(Z^i,y) \right).
	\end{aligned}
\end{equation}
\end{itemize}
\caption{Bayesian parameter identification} \label{alg1}
\end{algorithm}
Steps (3)--(6) can be realized consecutively or in parallel, dependening of the available computer architecture. 

Note that computed posterior moments \eqref{approxMoments} are random variables themselves. In particular, their values change for different samples $Z^i$ of $Z$. To demonstrate the validity of Theorem~\ref{thm:MSE} we repeat (3)--(6) multiple times 
to estimate the mean square error (MSE); see section \ref{subsect:numMSE} below. In the other examples we fit a particular probability density $\posterior$ to reproduce the posteriori moments \eqref{approxMoments} and discuss the outcomes. 

Clearly, the idealized local impedance boundary condition \eqref{IBC} does not always correctly describe the sound radiation and absorption at $\Gamma_R$. For significant model-data misfit, one would naturally expect failure of the Bayesian algorithm. In section \ref{subsect:3d-nonlocal} we demonstrate that the inconsistency of the parametric model is not always detected, i.e., for certain frequencies our algorithm finds values of $Z$ having high likelihood, even when the impedance condition \eqref{IBC} does not capture the physics at $\Gamma_R$. As we will see, the reasons for this somewhat unexpected behaviour are related to the actual microphone positions and the distribution of the eigenmodes of the physical model and their local structure.

It is natural that the demonstration of the effects related to consistency or inconsistency of the particular data sets $y = (y_1,\dots,y_m)$ and the parametric model \eqref{IBC} require full control over the generation and reproducibility of the noisy data $y$. That is why synthetic data generated from computer simulations are used in our experiments and not real measurements. We remark however, that our algorithm is designed to work with real measurements.

\subsection{On the selection of the prior and fitted posterior densities}
\label{subsect:numPrior}
As already mentioned above, the impedance parameter $Z = Z_R + iZ_I$ describes a passive absorbing surface. Mathematically this means $Z_R > 0$, whereas $Z_I$ is constraint-free. In other words, $Z$ can take values on the right complex half-plane. This justifies the selection of the log-normal prior density for $Z_R$ and the normal priori density for $Z_I$, that is
\begin{equation}\label{priorDensity}
\log Z_{R} \sim \mathcal{N}(\mu_{R},\gamma_{R}^{2}) \quad \text{ and } \quad Z_{I}\sim \mathcal{N}(\mu_{I},\gamma_{I}^{2})
\end{equation}
for the means $\mu_R$, $\mu_I$ and standard deviations $\gamma_R$, $\gamma_I$. Notice that according to \eqref{priorDensity} the real part $Z_R$ is not uniformly positive, as opposed to the requirement $Z_{R}\ge C >0$ imposed in the previous sections by assuming $Z\in U$, defined in \eqref{eq:U}. However, we have not detected any significant change in the performance of our algorithm, related to the presence of the uniform positivity condition, and therefore work with the model \eqref{priorDensity} in what follows. The joint probability density function is given by (see \cite{FlZu06})
\begin{align*}
\prior(Z_{R},Z_{I})  
= \frac{1}{2\pi \gamma_R\gamma_I} \frac{1}{Z_{R}} \exp\left(-\frac{(\log Z_{R}-\mu_{R})^2}{2 \gamma_R^{2}} -\frac{(Z_{I}-\mu_{I})^2}{2 \gamma_I^{2}} \right).
\end{align*}
Since the problem \eqref{eq:model_problem} has a unique solution of every $Z$, it is natural to expect that the posterior density will be unimodal if the data $y=(y_1,\dots,y_m)$ are consistent with the parametric impedance model \eqref{IBC}.  Similarly as $\prior$, the posterior density should be supported in the right half-plane. In view of this arguments, we approximate the posterior density $\posterior$ by
\begin{equation}\label{eq:posteriorhat}
\posteriorhat(Z_{R},Z_{I})  
= \frac{1}{2\pi \widehat{\gamma}_R\widehat{\gamma}_I} \frac{1}{Z_{R}} \exp\left(-\frac{(\log Z_{R}-\widehat{\mu}_{R})^2}{2 \widehat{\gamma}_R^{2}} -\frac{(Z_{I}-\widehat{\mu}_{I})^2}{2 \widehat{\gamma}_I^{2}} \right),
\end{equation}
where the parameters $\widehat{\mu}_R,\widehat{\mu}_I,\widehat{\gamma}_R,\widehat{\gamma}_I$ are fitted to reproduce the posterior moments \eqref{approxMoments} by

\begin{equation}\label{eq:postMoments}
\begin{aligned}
\widehat{\mu}_{R} &= \log \left(\frac{\widehat{\mathcal{M}}_{\posterior}^{1}[Z_{R}]^{2}_{h,N}}{\sqrt{\widehat{\mathcal{M}}_{\posterior}^{2}[Z_{R}]_{h,N}  +\widehat{\mathcal{M}}_{\posterior}^{1}[Z_{R}]_{h,N}^{2}}}\right),
&\widehat{\mu}_{I} &= \widehat{\mathcal{M}}_{\posterior}^{1}[Z_{I}]_{h,N},\\
\widehat{\gamma}_{R} &= \sqrt{\log\left(1+\frac{\widehat{\mathcal{M}}_{\posterior}^{2}[Z_{R}]_{h,N}}{\widehat{\mathcal{M}}_{\posterior}^{1}[Z_{R}]_{h,N}}\right)},
&\widehat{\gamma}_{I} &= \sqrt{\widehat{\mathcal{M}}_{\posterior}^{2}[Z_{I}]_{h,N}}.
\end{aligned}
\end{equation}
Note that all quantities with a hat $\widehat{\cdot}$ depend on the discretization and sampling parameters $h$ and $N$. For brevity, we will not explicitly track these parameters in the notation.  Although there is no indication that $\posteriorhat$ is close to $\posterior$, we will see in section \ref{subsect:numMSE} that the maximum of $\posteriorhat$ reproduces the reference value of $Z$ surprisingly well if the data $y=(y_1,\dots,y_m)$ are consistent with the impedance model \eqref{IBC}.

\subsection{2D model: Discretization and sampling error} \label{subsect:numMSE}
To validate the results from the previous sections we consider the model problem in a two-dimensional room of 3\,m width and 3.5\,m length, i.e., $D \coloneqq [0,3]\times[0,3.5]\subset\mathbb{R}^{2}$:
\begin{equation}
\label{eq:model_problem2D}
\left\{
\begin{split}
-\Delta p -k^2 p \quad=\quad f  &\quad\text{ in } D,\\
\frac{\partial p}{\partial n} + \frac{i\omega \rho}{Z^{(1)}} p \quad=\quad 0 &\quad\text{ on } \Gamma_{R}^{(1)},\\
\frac{\partial p}{\partial n} + \frac{i\omega \rho}{Z^{(2)}} p \quad=\quad 0 &\quad\text{ on } \Gamma_{R}^{(2)},\\
\frac{\partial p}{\partial n} \quad=\quad 0 &\quad\text{ on } \Gamma_{N}.
\end{split}
\right.
\end{equation}
Here, the boundary $\partial D$ of $D$ is split into three disjoint parts $\Gamma_{R}^{(1)} = \{ (x_{1},x_{2}) \in \partial D: x_{1} = 0\}$, $\Gamma_{R}^{(2)} = \{ (x_{1},x_{2}) \in \partial D: x_{2} = 0\}$  and $\Gamma_{N} = \partial D \backslash(\Gamma_{R}^{(1)}\cup \Gamma_{R}^{(2)})$, see figure \ref{fig:2d}. We assume that the acoustic impedances $Z^{(1)} = Z_{R}^{(1)} + iZ_{I}^{(1)}$ and $Z^{(2)} = Z_{R}^{(2)} + iZ_{I}^{(2)}$, $Z_{R}^{(\ell)}, Z_{I}^{(\ell)}\in \mathbb{R}$ for $\ell = 1,2$, are constant on the specific part of the boundary. The point source $f=\delta^{\var}$ is located at $s=(1,1)\in D_{\kappa}$. The data are computed by numerically solving problem \eqref{eq:model_problem2D} with reference acoustic impedance values, namely $Z_{\mathrm{ref}}^{(1)} =400-700i$ and $Z_{\mathrm{ref}}^{(2)} = 500+800i$, evaluating the solution at $m=4$ measurement positions and adding artificial noise, i.e., $y = \mathcal{G}(Z_{\mathrm{ref}}) + \eta$, where $\eta \sim \mathcal{CN}(0,\Sigma,0)$. $\Sigma$ is a diagonal matrix with entries $\sigma_{k}^{2}$, $k=1,\dots,2m$, where $\sigma_{k}=0.02$ for all $k$. This means that the log-likelihood can be expected to be of the order of $-M$ if the sampling and discretization errors are neglected, since then $\mathbb{E}_{\prior}[\log\theta] = \mathbb{E}_{\prior}[-\frac{1}{2} \lVert \eta \rVert_{\Sigma}^{2}] = M$. The prior distribution was chosen as described in section \ref{subsect:numPrior} with $\log Z_{R}^{(1)} \sim \mathcal{N}(300,200^2)$, $Z_{I}^{(1)} \sim \mathcal{N}(-600,200^2)$, $\log Z_{R}^{(2)} \sim \mathcal{N}(600,200^2)$ and $Z_{I}^{(2)} \sim \mathcal{N}(900,200^2)$.

To confirm the theoretical convergence rates from Theorem \ref{thm:MSE} we compute $\widehat{M}^{1}_{\posterior}[Z_{*}^{(\ell)}]$ and $\widehat{M}_{\posterior}^{2}[Z_{*}^{(\ell)}]$, i.e., we use $\phi(Z_{*}^{(\ell)})=Z_{*}^{(\ell)}$ or $\phi(Z_{*}^{(\ell)}) = (Z_{*}^{(\ell)}-\widehat{M}_{\posterior}^{1}[Z_{*}^{(\ell)}])^2 $ in Theorem \ref{thm:MSE}. These let us compute the estimates for $\widehat{\mu}_{R}^{(\ell)}$,$\widehat{\gamma}_{R}^{(\ell)}$, $\widehat{\mu}_{I}^{(\ell)}$ and $\widehat{\gamma}_{I}^{(\ell)}$ with equations \eqref{eq:postMoments}.

First, consider the discretization error, i.e., $\lvert \mu_{*}^{(\ell)} - \widehat{\mu}_{*}^{(\ell)}\rvert$ and $\lvert \gamma_{*}^{(\ell)} - \widehat{\gamma}_{*}^{(\ell)}\rvert$, $\ell=1,2$, where we recall that  $\widehat{\mu}_{*}^{(\ell)}$ and $\widehat{\gamma}_{*}^{(\ell)}$ depend on the discretization parameters $h,N$. For the estimation $N=2^{16}$ samples are drawn. The exact expected values $\mu_{*}^{(\ell)}$ and $\gamma_{*}^{(\ell)}$, $\ell=1,2$, are approximated on a small grid with $h=\frac{c/f}{10 \cdot 2^5}\approx 0.021$\,m for reference. The error visualized in figure \ref{fig:DiscretizationError} is averaged over 20 runs, i.e., different set of drawn samples, while the positions of the measurements and the noise (and therefore the data $y$) stay fixed for all runs. As seen from figure \ref{fig:DiscretizationError}, the convergence behaviour confirms the theoretically predicted convergence rates in Theorem \ref{thm:MSE}.\\

\begin{figure}[htbp]
\centering
\includegraphics[width=1\textwidth]{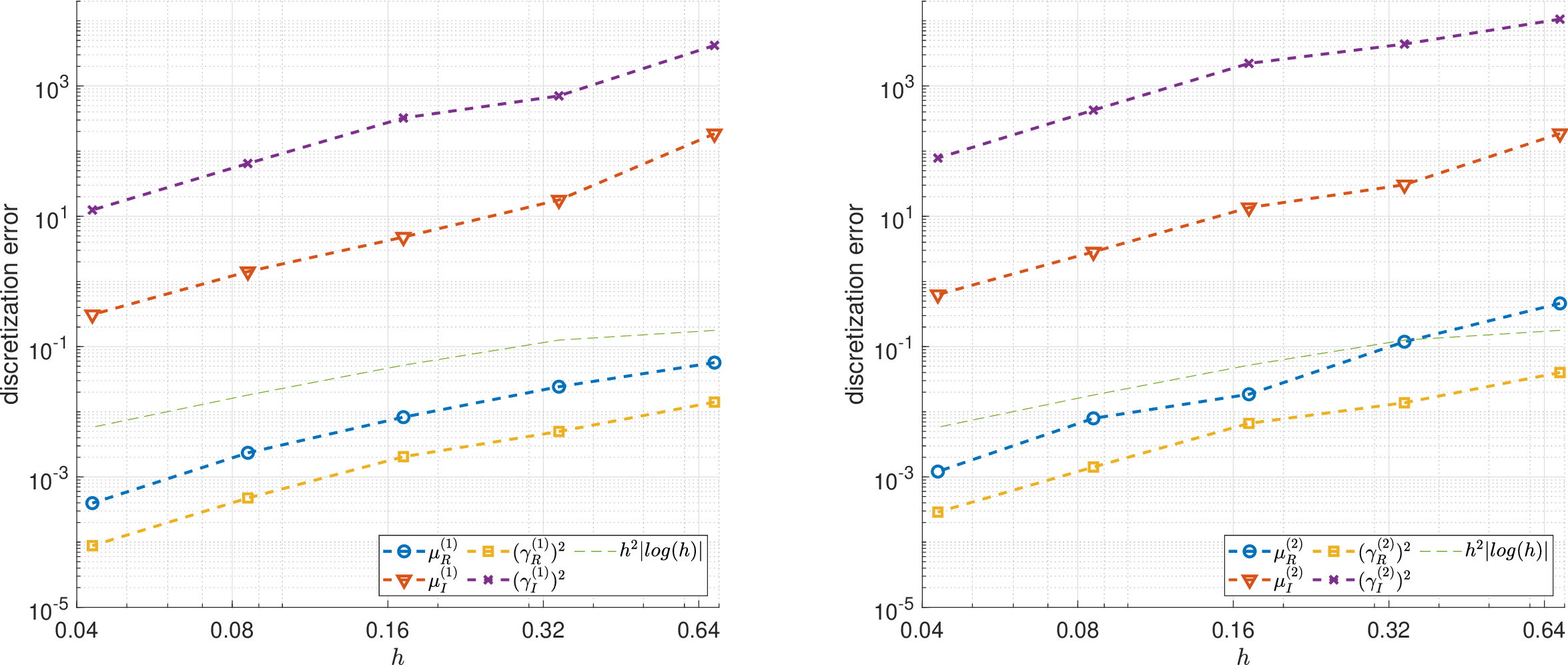}
\caption{Discretization error for statistical parameters of 
$Z^{(1)}$ (left) and $Z^{(2)}$ (right) for $f=50\,$Hz.}
\label{fig:DiscretizationError}
\end{figure}

Next, we address the sampling error. The exact values are approximated with large sample size of $N=2^{16}$. The computations for the sampling error were all done on a grid with mesh size $h=\frac{c/f}{20}\approx 0.34$\,m and frequency $f=50$\,Hz. In figure \ref{fig:SamplingError} the sampling error is averaged over 20 runs. We observe the convergence rate of $\frac{1}{\sqrt{N}}$ as expected from Theorem \ref{thm:MSE}.

\begin{figure}[htbp]
\centering
\includegraphics[width=0.9\textwidth]{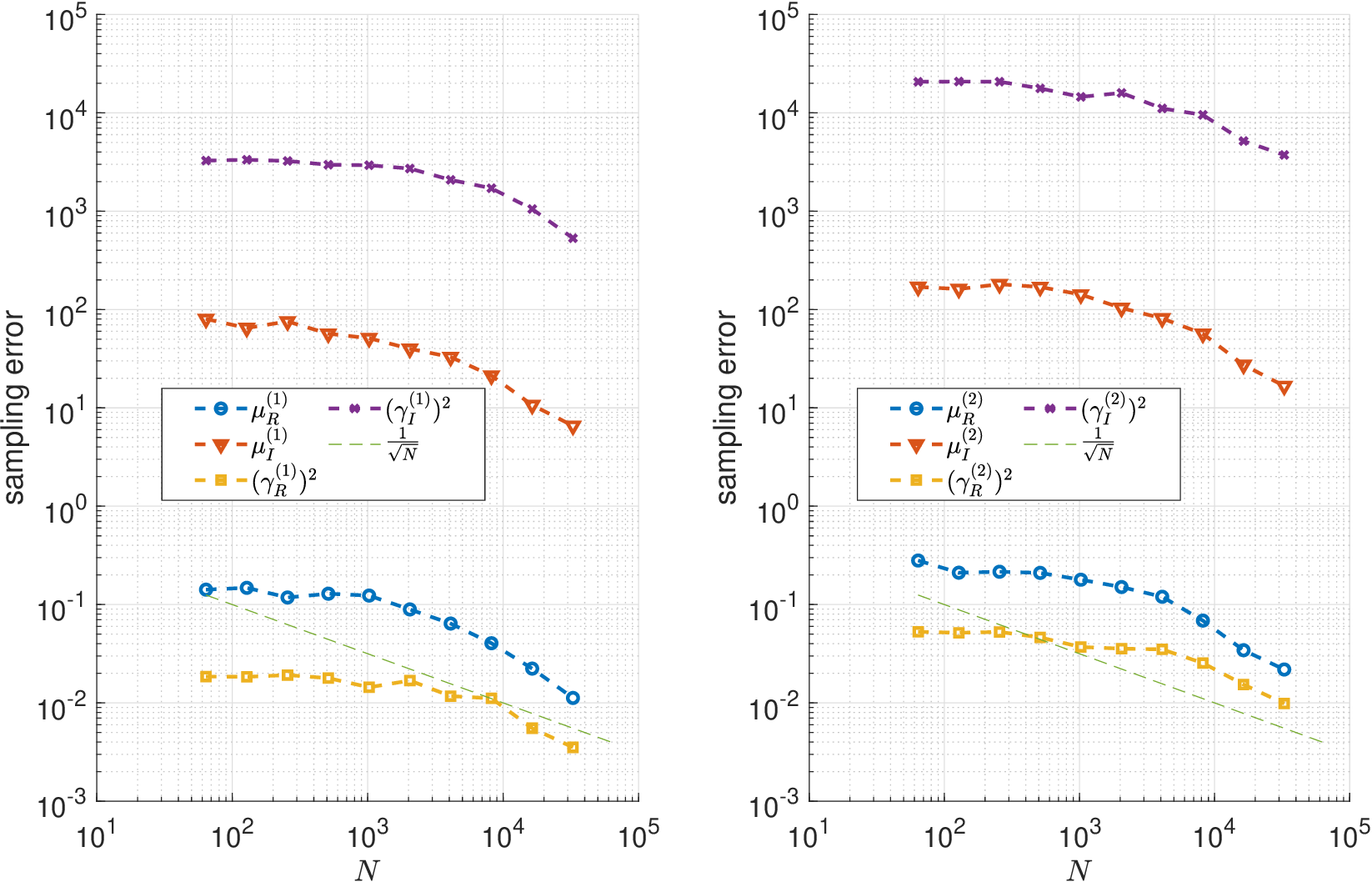}
\caption{Sampling error for statistical parameters of 
$Z^{(1)}$ (left) and $Z^{(2)}$ (right) for $f=50\,$Hz.}
\label{fig:SamplingError}
\end{figure}
\begin{figure}[htbp]
\centering
\includegraphics[width=1\textwidth]{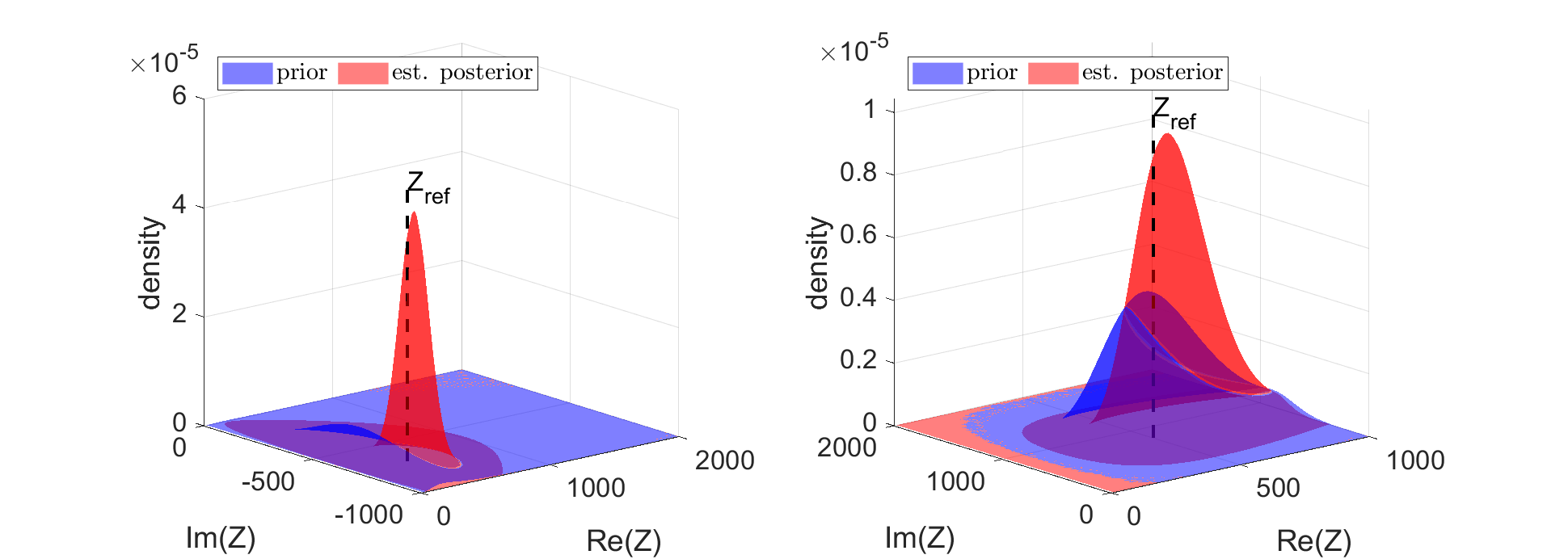}
\caption{Prior density and density given by estimated expected value and standard deviation for $f=50$\,Hz for $Z^{(1)}$ (left) and $Z^{(2)}$ (right).}
\label{fig:density2D_1}
\end{figure}

\begin{figure}[htbp]
\centering
\includegraphics[width=1\textwidth]{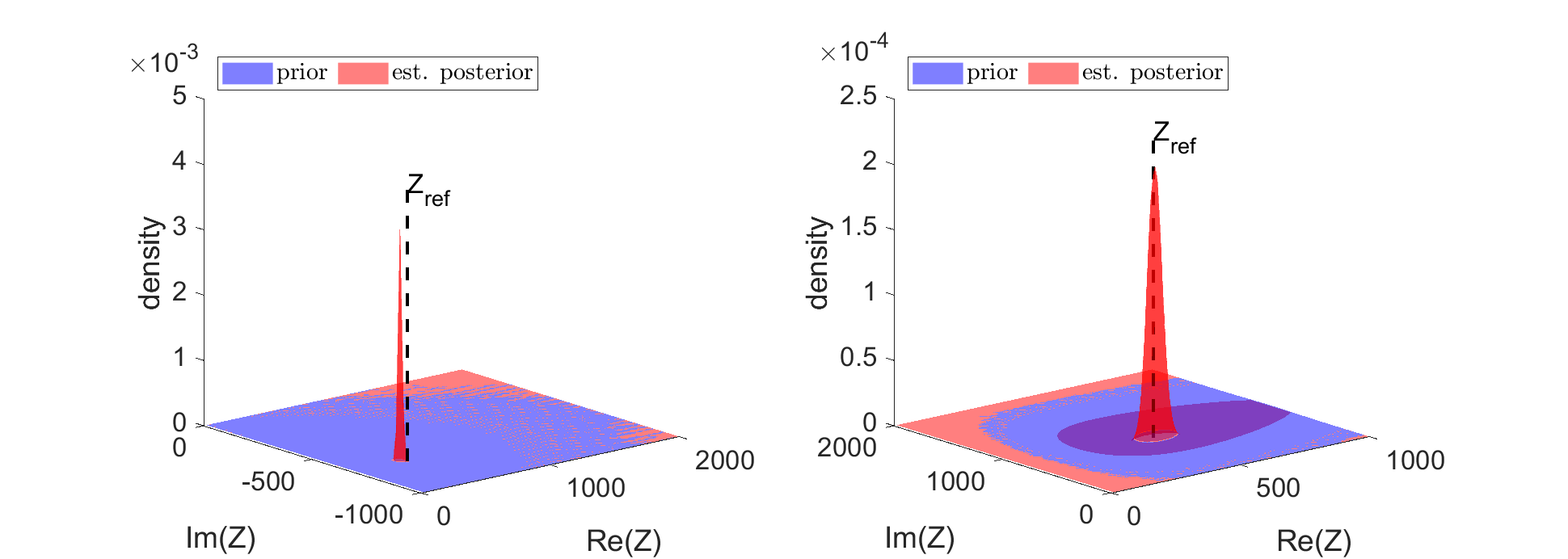}
\caption{Prior density and density given by estimated expected value and standard deviation for $f=100$\,Hz for $Z^{(1)}$ (left) and $Z^{(2)}$ (right).}
\label{fig:density2D_2}
\end{figure}

Given the estimations on the first and second moments of $Z_{R}^{(\ell)}$ and $Z_{I}^{(\ell)}$, $\ell=1,2$, we can approximate the posterior density $\posterior$ by $\posteriorhat$ defined in equation \eqref{eq:posteriorhat} as described in section section \ref{subsect:numPrior}. For $N=2^{16}$ samples and $h=\frac{c/f}{20}$ we again compute the estimators given in \eqref{eq:postMoments}. In figures \ref{fig:density2D_1} and \ref{fig:density2D_2} we compare the approximate posterior density $\posteriorhat$ with the prior density for $50$\,Hz and $100$\,Hz, respectively. We observe that the peak of the density function $\posteriorhat$ is closer to the true underlying value than the peak of $\prior$, and the variance reduces. Consequently, we obtain a sharper estimate for the impedance than the initial prior. Instead of only pointing out the single best value for the impedance fitting the data, we get a density plot that shows a whole domain of reasonable parameter combinations under the assumption that certain noise is present.

Note that for the real part, the peak of the density is lower than the expected value as it should be for the lognormal distribution. Thus, the expected value as a single parameter is rather insufficient to provide the best estimate for the impedance: the expected value and the standard deviation, or a fitted distribution, as we propose here, result in a more complete and informative estimate.

Our result can be optionally utilized in several ways to design further simulations (e.g., acoustic modeling of another room with a wall made from the same material, or a room with a set of acoustic obstacles inside it). If the aim is to use a single deterministic value of the estimated acoustic impedance that best fits the data, one should rather use the value that maximizes the likelihood. If the focus is in the further probabilistic parameter modeling (e.g., further Bayesian simulations), the estimated statistical moments and the fitted distribution can be used instead.
\begin{figure}[htbp]
	\centering
	\includegraphics[width=1\textwidth]{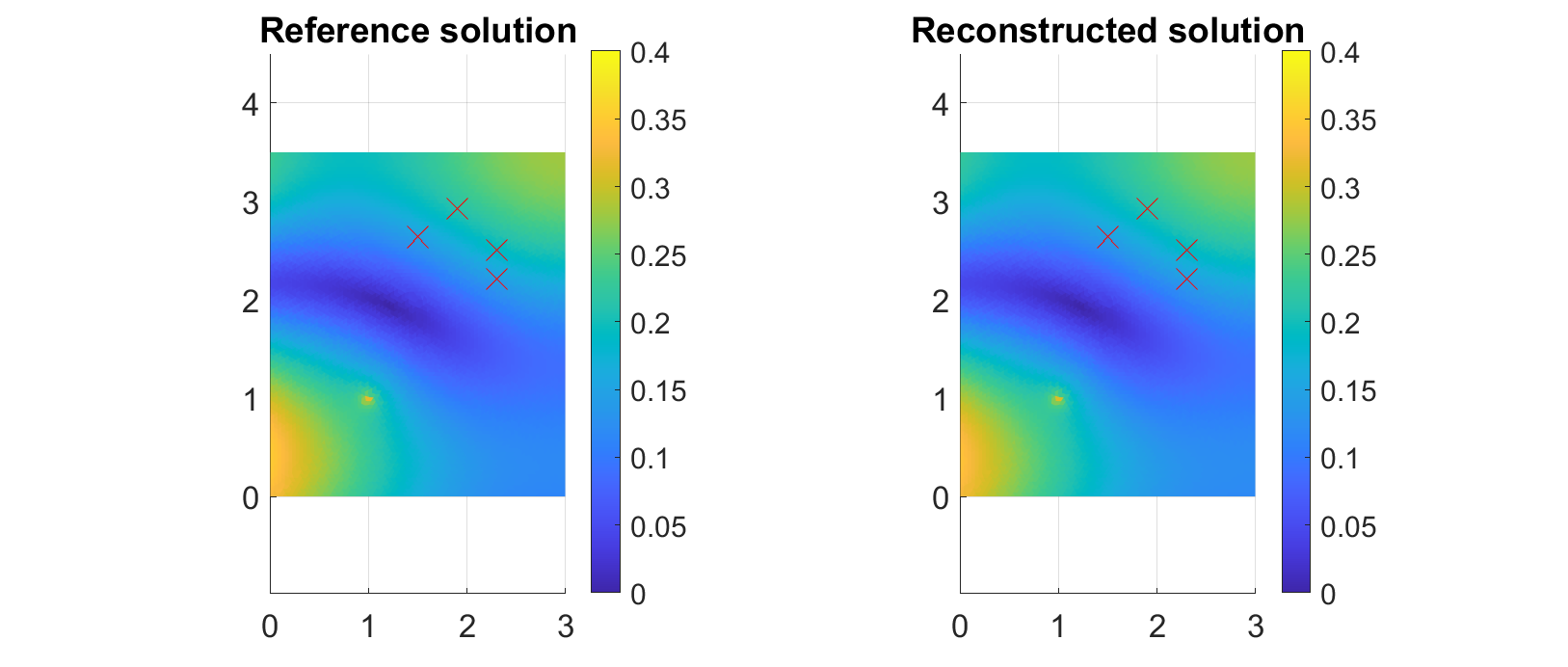}
	\caption{ Absolute value of the reference solution with $Z_{\mathrm{ref}}^{(1)} = 400-700i$ and $Z_{\mathrm{ref}}^{(2)}=500+800i$ (left) and the reconstructed solution with the most likely acoustic impedance value during the sampling process $Z^{(1)} \approx 427-697i $ and $Z^{(2)} \approx 583+860i $ (right) for $f=50$\,Hz. The blue dot shows the position of the source. The red marks are the measurement positions.}
	\label{fig:data2D}
\end{figure}
Figure \ref{fig:data2D} shows an example of solution plots with marked source and receiver positions for $f=50\,$Hz. The solution plot the data was generated from, i.e., the solution to problem \eqref{eq:model_problem2D} using the reference values $Z_{\mathrm{ref}}^{(1)} = 400-700i$ and $Z_{\mathrm{ref}}^{(2)}=500+800i$ in the left panel, looks very similar to the simulation with the estimated impedance value $\widehat{Z}$ in the right panel. The latter was chosen as the sample $Z^{i}$ for which the likelihood $\theta_{h}$ was largest, i.e., 
\begin{equation*}
\widehat{Z} \coloneqq \argmax\limits_{Z^{i}, i=1,\dots,N} \theta_{h}(Z^{i},y).
\end{equation*}
\subsection{3D problem with data from impedance problem}\label{sec:3Dimp}
We again consider the model problem \eqref{eq:model_problem2D}. However this time we consider a room in three dimensions with width 3\,m, length 3.5\,m and height 2.5\,m, i.e., $D \coloneqq [0,3]\times[0,3.5]\times[0,2.5]\subset{\mathbb{R}^{3}}$, $\Gamma_{R}^{(1)} = \{ (x_{1},x_{2},x_{3}) \in \partial D: x_{1} = 0\}$,  $\Gamma_{R}^{(2)} = \{ (x_{1},x_{2},x_{3}) \in \partial D: x_{2} = 0\}$ and $\Gamma_{N} = \partial D \backslash(\Gamma_{R}^{(1)}\cup \Gamma_{R}^{(2)})$, see figure \ref{fig:3D}. Here $Z^{(1)} = 500-800i$ is assumed to be known and only moments of an unknown $Z^{(2)}$ are computed. The point source $f=\delta^{\var}$ is located at $s=(1,1,1)\in D_{\kappa}$.
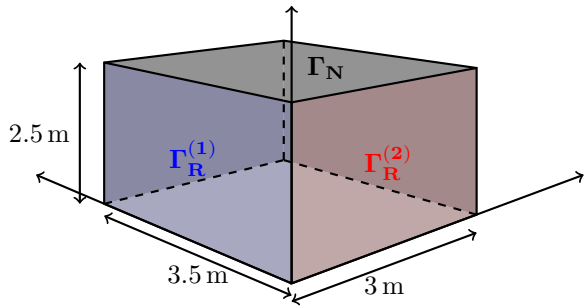
\begin{figure}[htbp]
\centering
\begin{tikzpicture}[scale=1.6]
\coordinate (P1) at (-7cm,1.5cm);
\coordinate (P2) at (8cm,1.5cm);

\coordinate (A1) at (0em,0cm);
\coordinate (A2) at (0em,-1.5cm);

\coordinate (A3) at ($(P1)!.78!(A2)$);
\coordinate (A4) at ($(P1)!.78!(A1)$);

\coordinate (A7) at ($(P2)!.81!(A2)$);
\coordinate (A8) at ($(P2)!.81!(A1)$);

\coordinate (A5) at (intersection cs: first line={(A8) -- (P1)}, second line={(A4) -- (P2)});
\coordinate (A6) at (intersection cs: first line={(A7) -- (P1)}, second line={(A3) -- (P2)});

\fill[gray!60] (A2) -- (A3) -- (A6) -- (A7) -- cycle;
\fill[gray!90] (A3) -- (A4) -- (A5) -- (A6) -- cycle;
\fill[gray!90] (A6) -- (A5) -- (A8) -- (A7) -- cycle;
\fill[blue!50, opacity=0.2] (A1) -- (A2) -- (A3) -- (A4) -- cycle;
\fill[gray!60, opacity=0.2] (A1) -- (A4) -- (A5) -- (A8) -- cycle;
\fill[red!50, opacity=0.2] (A2) -- (A7) -- (A8) -- (A1) -- cycle; 

\draw[thick,dashed] (A3) -- (A6);
\draw[thick,dashed] (A5) -- (A6);
\draw[thick,dashed] (A7) -- (A6);

\draw[thick] (A2) -- (A3) -- (A4) -- (A5) -- (A8) -- (A7) -- (A2) -- (A1) -- (A8);
\draw[thick] (A1) -- (A4);


\draw[thick,arrows=<->] ($(A3)-(0.2cm,0)$) -- node[left] {2.5\,m} ($(A4)-(0.2cm,0)$);
\draw[thick,arrows=<->] ($(A2)-(0,0.15cm)$) -- node[below] {3\,m} ($(A7)-(0,0.15cm)$);
\draw[thick,arrows=<->] ($(A3)-(0,0.1cm)$) -- node[below] {3.5\,m} ($(A2)-(0,0.1cm)$);

\node at (barycentric cs:A1=1,A2=1,A3=1,A4=1.2) {\textcolor{blue}{$\mathbf{\Gamma_{R}^{(1)}}$}};
\node at (barycentric cs:A1=1,A2=1,A7=1,A8=1.2) {\textcolor{red}{$\mathbf{\Gamma_{R}^{(2)}}$}};
\node at (barycentric cs:A1=1,A4=1,A5=1,A8=2) {$\mathbf{\Gamma_{N}}$};

\draw[thick,arrows=->] (A2) --  ($(P2)!.7!(A2)$) node[right] {$ $};
\draw[thick,arrows=->] (A2) --  ($(P1)!.7!(A2)$) node[left] {$ $};
\draw[thick,arrows=->] (A2) --  ($(A1)+(0,0.8cm)$) node[above] {$ $};
\end{tikzpicture}
\caption{Domain for model problem in 3D.}
\label{fig:3D}
\end{figure}

\begin{figure}[htbp]
\centering
\includegraphics[width=1\textwidth]{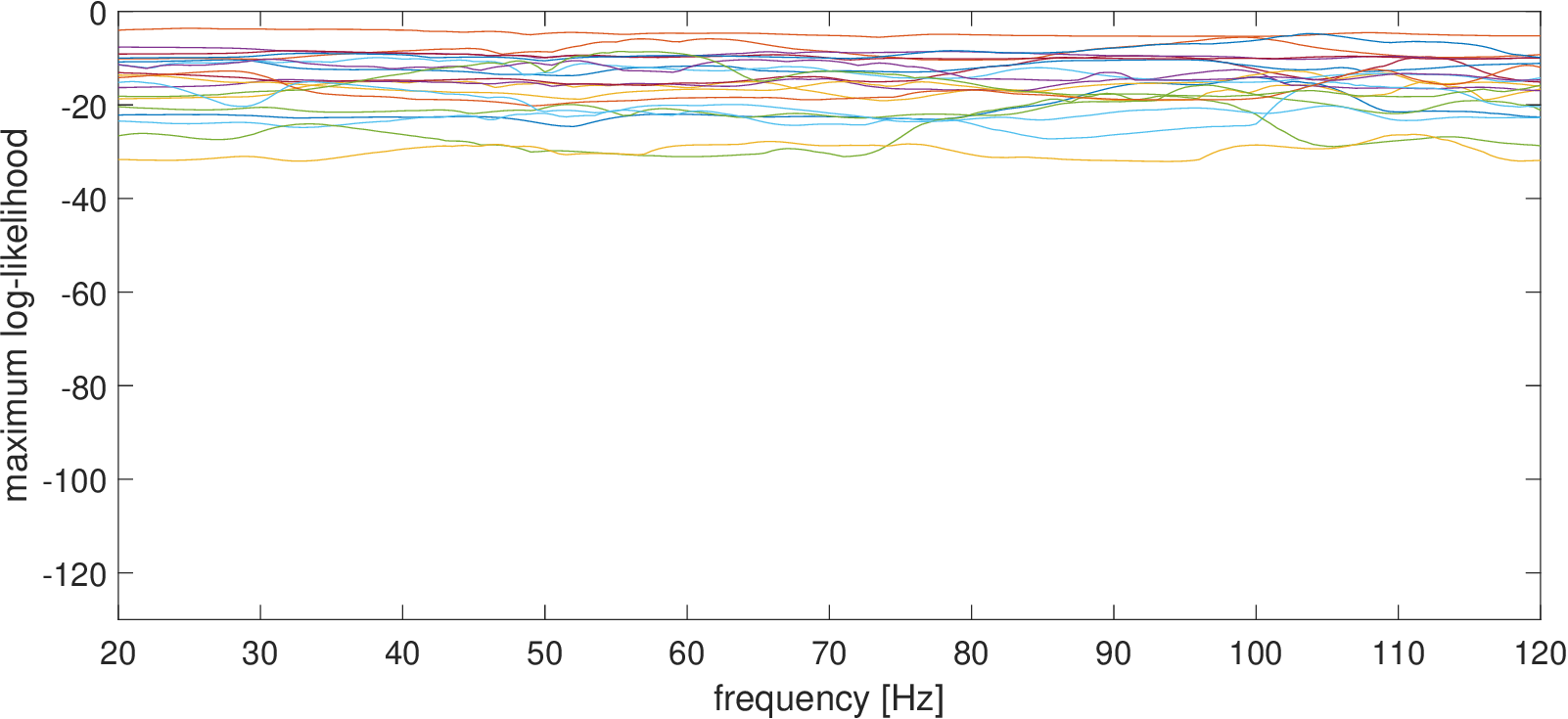}
\caption{Maximum log-likelihood for the impedance problem. The thin lines correspond to the individual runs.}
\label{fig:maxLikelihoodImpedance}
\end{figure}

\begin{figure}[htbp]
\centering
\includegraphics[width=1\textwidth]{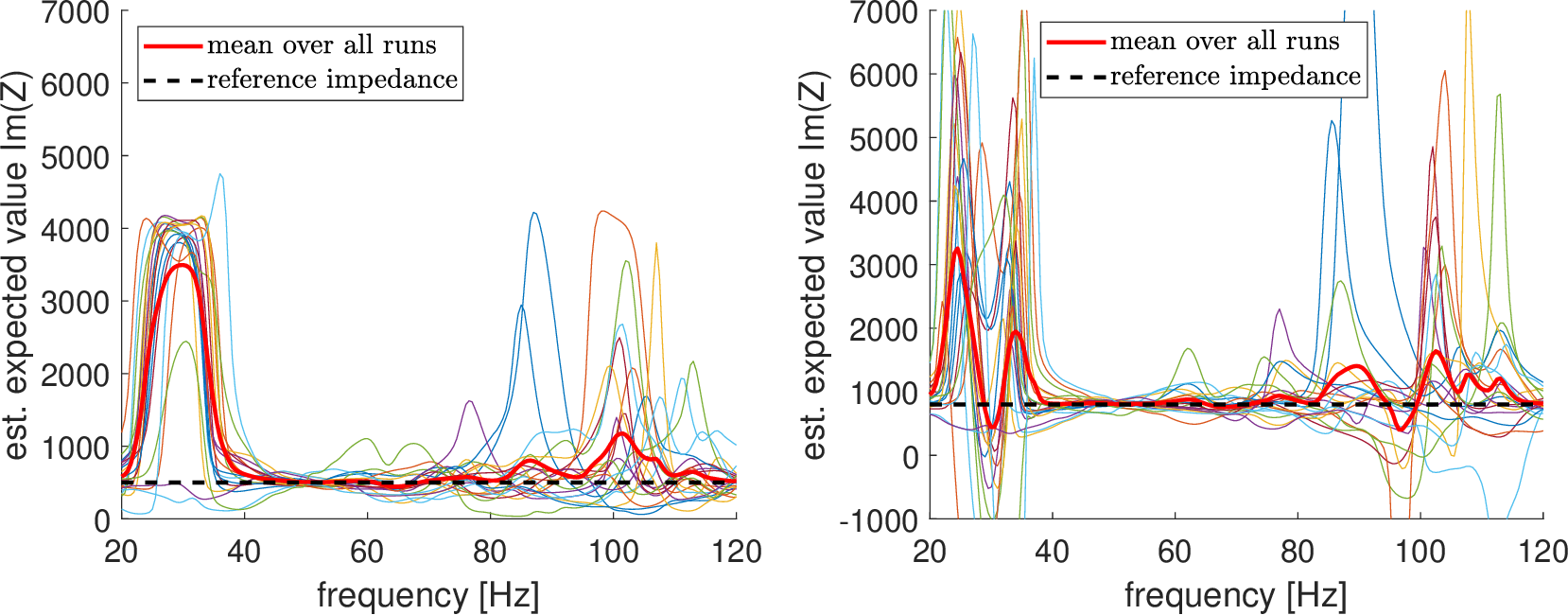}
\caption{Estimated expected value (real part left, imaginary part right) of the impedance for the impedance problem for all 20 runs. The black dashed line indicates the reference impedance value. The red line is the mean over all runs. Thin lines correspond to individual runs.}
\label{fig:estimatedMu}
\end{figure}

\begin{figure}[htbp]
\centering
\includegraphics[width=1\textwidth]{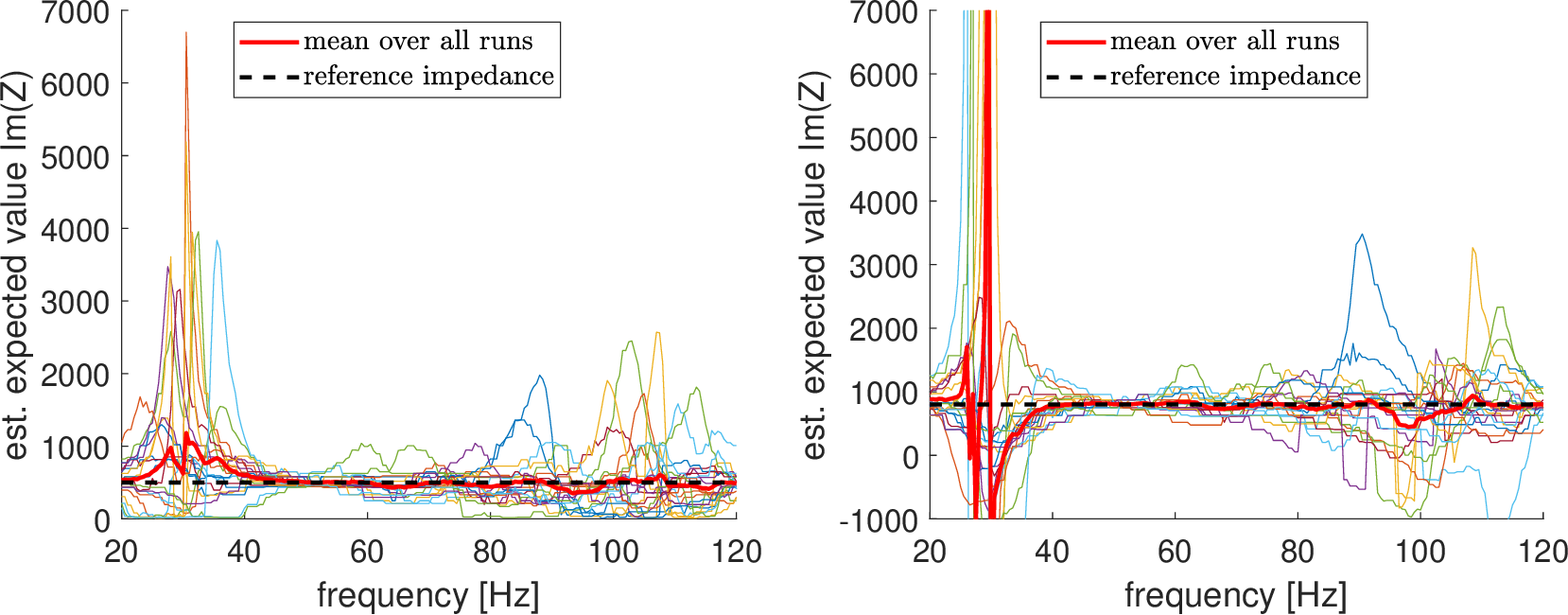}
\caption{Most likely impedance (real part left, imaginary part right) for the impedance problem  for all 20 runs. The black dashed line indicates the reference impedance value. The red line is the mean over all runs. Thin lines correspond to individual runs.}
\label{fig:MLreal}
\end{figure}

\begin{figure}[htbp]
\centering
\includegraphics[width=0.95\textwidth]{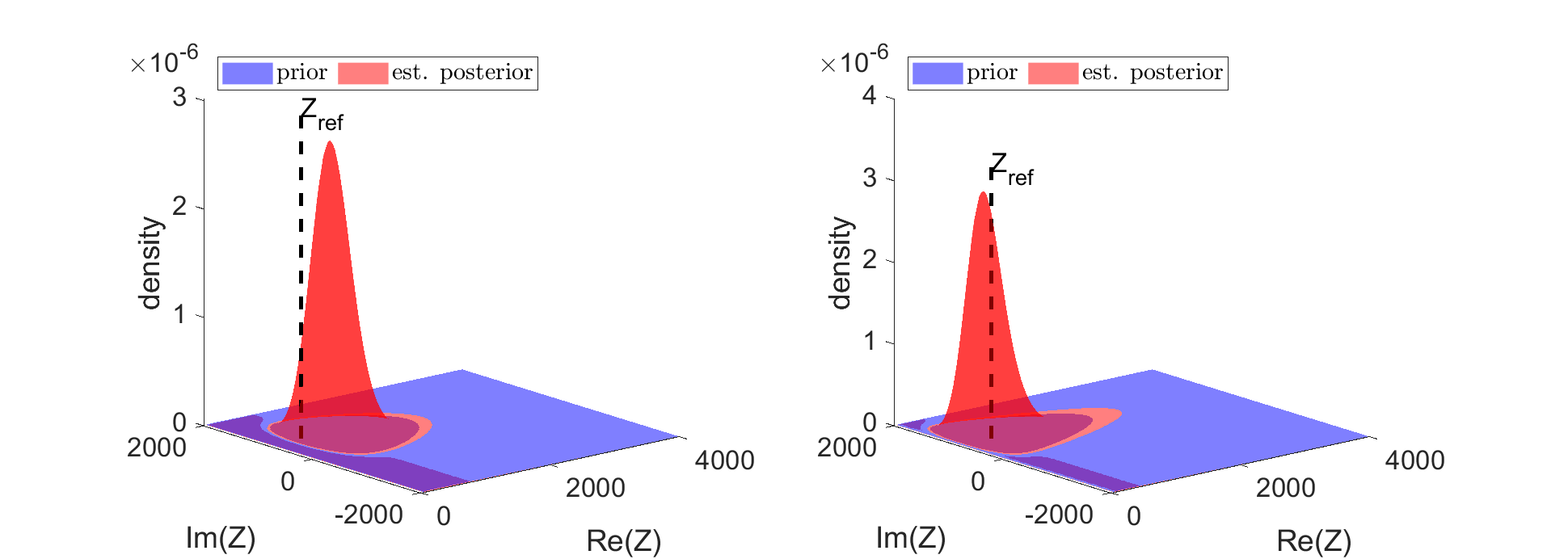}
\caption{Prior density and density given by estimated expected value and standard deviation for $f=25$\,Hz (left) and $f=95$\,Hz (right).}
\label{fig:density1}
\end{figure}

In the experiments the artificial measurements $y$ are generated by solving problem \eqref{eq:model_problem2D} with $Z^{(2)}=Z_{\mathrm{ref}}=500+800i$, evaluating the solution at $m=16$ measurement positions and adding noise, i.e., $y=\mathcal{G}(Z_{\mathrm{ref}})+\eta$, where $\eta \sim \mathcal{CN}(0,\Sigma,0)$. Here $\Sigma$ is a diagonal matrix with entries $\sigma_{k}^{2}=\sigma_{0}^{2}$, $k=1,\dots,2m$. The variance of the noise is given by $\sigma_{0}^2 = 0.02$. For each run   $m$ measurement positions are randomly placed in $M_{\kappa}$ with $\kappa = \frac{1}{2}$, i.e., in contrast to the 2D experiments in the previous section, the receiver positions and noise change from run to run. This allows us to observe behavior that is not depending on specific measurement positions. The prior distribution is chosen as described in section \ref{subsect:numPrior} with $\log Z_{R}^{(2)} \sim \mathcal{N}(4000,10000^2)$ and $Z_{I}^{(2)} \sim \mathcal{N}(0,30000^2)$.

For frequencies from $20$\,Hz to $120$\,Hz in steps of $0.5$\,Hz we estimate the expected value and variance for the posterior of $Z_{R}^{(2)}$ and $Z_{I}^{(2)}$ in the same way as outlined above. Here we take $N=2^{14}$ samples and choose the mesh width as $h=\min(\frac{c/f}{20},0.5)$. We observe that parameters with high likelihoods are detected. In figure \ref{fig:maxLikelihoodImpedance} we show the maximum of the log-likelihood achieved from samples during the Monte Carlo sampling, i.e., $\max \theta_{h}(Z^{i},y)$. The estimated expected values, i.e., $\widehat{\mathcal{M}}^1_y[Z_{*}^{(2)}]_{h,N}$, are close to the true underlying parameter values of $Z^{(2)}$ for most frequencies, see figure \ref{fig:estimatedMu}. In the lower frequency range of 20\,Hz to 40\,Hz we observe rather large fluctuations around the reference value. In this frequency range the solution of \eqref{eq:model_problem2D} seems to be less dependent on the acoustic impedance $Z^{(2)}$ and hence impedance values $Z^{(2)}$ with large magnitude still lead to high likelihood values. The samples $Z^i$ with largest likelihood, i.e., 
\begin{equation*}
\widehat{Z} \coloneqq \argmax\limits_{Z^{i}, i=1,\dots,N} \theta_{h}(Z^{i},y),
\end{equation*}
the best fitting parameter-pairs, are visualized in figure \ref{fig:MLreal}. For the low frequency range we observe that for some runs the impedance is estimated badly, however we see good approximation for most runs. The prior density $\prior$ together with the estimated posterior density $\posteriorhat$ given by the estimated parameters are visualized for two selected frequencies in figure \ref{fig:density1}. The prior density looks flat, since its variance is large. The estimated posterior, however, has smaller variance and leads to a suitable estimate of the reference impedance.

\subsection{3D problem with data from coupled acoustic-structural problem} \label{subsect:3d-nonlocal}
In this section we again consider the 3D version of the problem \eqref{eq:model_problem2D} as before, but this time the data was generated in COMSOL MULTIPHYSICS by solving a coupled acoustic-structural problem, where instead of an impedance boundary condition at $\Gamma_{R}^{(2)}$ a model of a glass wall is considered. In particular, this means that this part of the boundary is not locally reacting. 

The model to estimate moments of the posterior distribution remains the pure impedance problem as before.  Hence the model used to approximate the data are different from the one that is used to generate the data, and thus, leads to a significant \emph{model-data misfit}. In this section we investigate how well the data can be estimated by the locally reacting impedance model \eqref{IBC} and whether the inconsistency of the model and the data can be detected by the simulations. 
The remaining experimental setup including the domain, prior, noise, source position and microphone positions are the same as in section \ref{sec:3Dimp}. In fact, the same samples $Z^{i}$ drawn in the previous subsection and the corresponding observations $\mathcal{G}_{h}(Z^{i})$ can be recycled here again. Only the data $y$ changes and hence only the likelihood $\theta_{h}$ has to be evaluated again to compute the approximate moments in \eqref{eq:postMoments}.\\

Inspecting the maximum log-likelihood in the frequency range 20-120\,Hz in figure \ref{fig:maxLikelihoodCoupledWef}, we observe that for some frequencies its value gets very negative for \emph{all} samples $Z^{i}$ in each of the 20 runs. This is the consequence of the model-data misfit.
Nevertheless for some frequency ranges the algorithm detects well fitting impedance values from the parametric model.

\begin{figure}[htbp]
\centering
\includegraphics[width=0.95\textwidth]{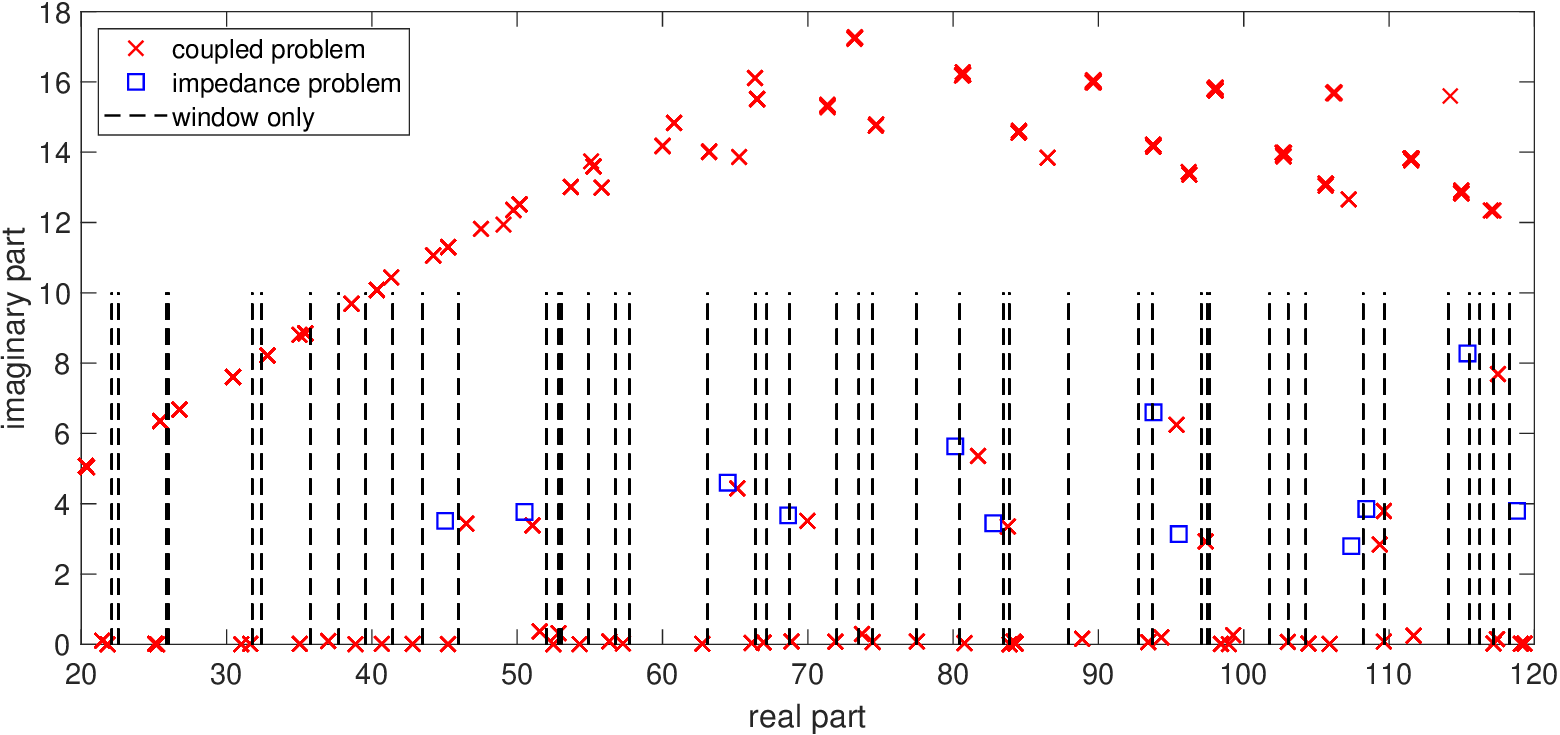}
\caption{Eigenfrequencies of coupled problem, pure impedance problem (sound hard at $\Gamma_{R}^{(2)}$) and for the glass wall (window) decoupled from the room acoustics (boundary conditions were set to be free at the large front and back surfaces and fixed at the edges of the window).}
\label{fig:ef}
\end{figure}
\begin{figure}[htbp]
\centering
\includegraphics[width=0.95\textwidth]{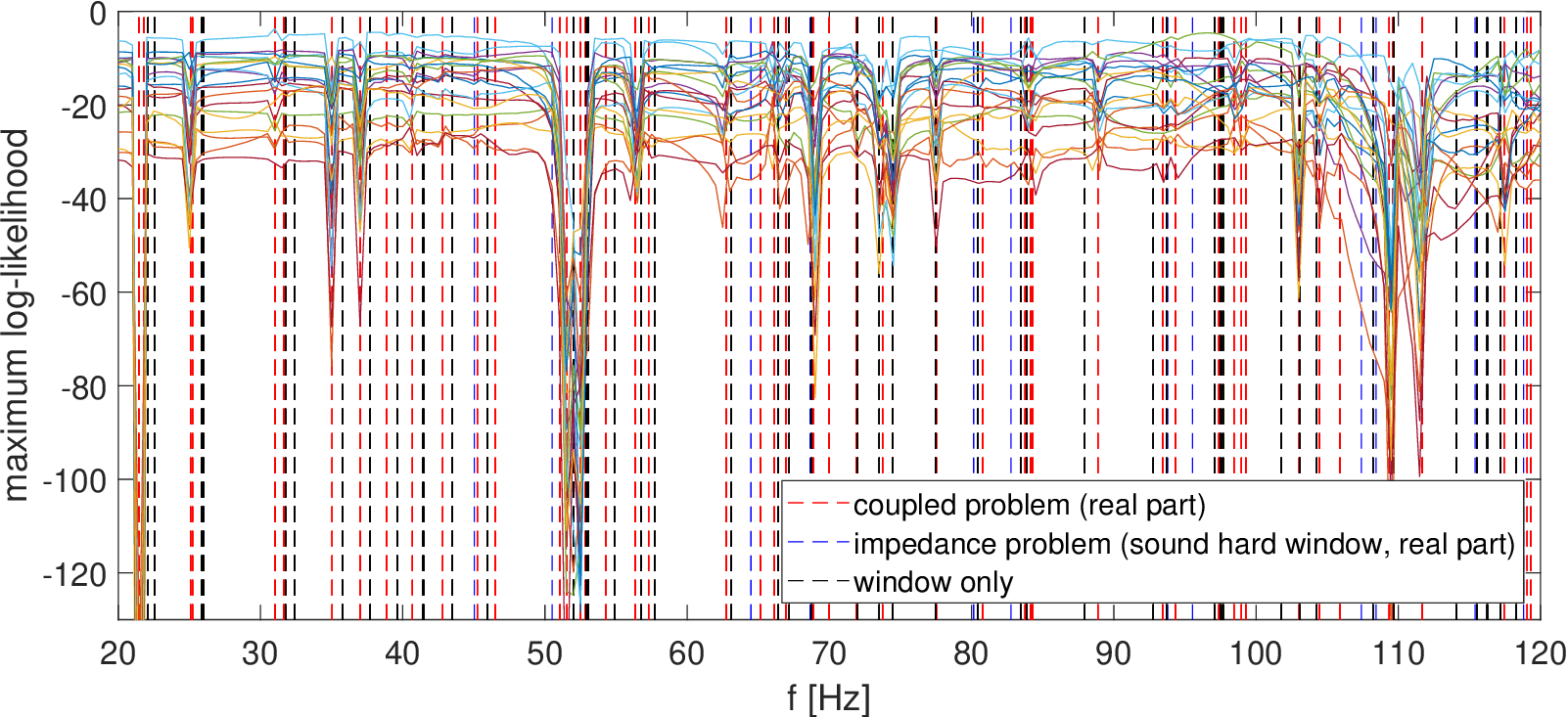}
\caption{Maximum log-likelihood for the coupled problem with eigenfrequencies (real part). The thin lines correspond to the individual runs.}
\label{fig:maxLikelihoodCoupledWef}
\end{figure}

To further study this behavior we compute the eigenfrequencies for the coupled problem, the glass wall itself and the impedance problem \eqref{eq:model_problem2D} (with sound hard $\Gamma_{R}^{(2)}$, i.e., for $\lvert Z^{(2)} \rvert \rightarrow \infty$). The eigenfrequencies are shown in figure \ref{fig:ef}. Notice, that for the glass wall the eigenfrequencies are real, but for the discretized coupled and impedance problem the eigenfrequencies are complex-valued with small imaginary part due to the boundary condition at $\Gamma_{R}^{(1)}$. Comparing the real part of the eigenfrequencies of the coupled problem with those from the glass wall itself we observe that their locations coincide quite well with only small shifts, in particular the real parts of eigenfrequencies of the impedance problem are only slightly smaller than the ones from the coupled problem. Comparing the eigenfrequencies with small imaginary part, i.e., those close to purely real frequencies, with the maximum likelihood, see figure \ref{fig:maxLikelihoodCoupledWef}, we observe the following: If the likelihood vanishes there are also eigenfrequencies nearby. If the frequency is far from any eigenfrequency the coupled problem behaves similar to the impedance problem. We also notice that for some eigenfrequencies the log-likelihood $\log\theta_{h}$ is still of a similar magnitude as in the case of section \ref{sec:3Dimp}. 

In Figures \ref{fig:data1}--\ref{fig:data3} we demonstrate the pressure (in dB) for different frequencies for the reference solution of the coupled problem (i.e., the solution from which the artificial data have been generated), and the solution of \eqref{eq:model_problem2D}, where the value of the impedance $Z^{(2)}$ is the sample having the highest likelihood $\theta_{h}$.
To illustrate three typical types of behaviour over the frequency range 20--120 Hz we focus on the three representative values
$f=51.5$\,Hz, $f=54.5$\,Hz and $f=69$\,Hz. As seen from figure \ref{fig:maxLikelihoodCoupledWef}, these frequencies are close to the eigenfrequencies of the coupled problem. In general, we observe  that for some eigenfrequencies the solution in the interior is much more affected by the behaviour of the glass wall than for others. 

For $f=51.5$\,Hz a very low maximum log-likelihood was observed for all runs, see figure \ref{fig:maxLikelihoodCoupledWef}. In the left panel of figure \ref{fig:data1} we can clearly see the eigenmodes at the glass wall and that the reference solution in the interior of the domain is affected by its behaviour. In contrast, the solution plot for the locally reacting impedance boundary condition for the likelihood maximizer does not show oscillations for this frequency value, see the right panel of figure \ref{fig:data1}. Therefore, the simulation is not capable to fit the data and the model-data misfit can be clearly detected from the low values of the likelihood. 

For $f=54.5$\,Hz we again observe the eigenmodes on the glass wall for the reference solution in the left panel of figure \ref{fig:data2}. However this time the algorithm is able to determine impedance parameters that fit the data well and result in high values of the likelihood. A closer inspection of the reference solution and the likelihood maximizer in the left and right panels of figure \ref{fig:data2} show the similar structure in the interior of the domain, whereas the boundary behaviour is quite different. Here, the glass wall eigenmode is quite localized near the boundary, so that the oscillations cannot be detected by the microphones located in the interior of the domain. In this case, the algorithm does not detect the model-data misfit.

\begin{figure}[htbp]
\centering
\includegraphics[width=1\textwidth]{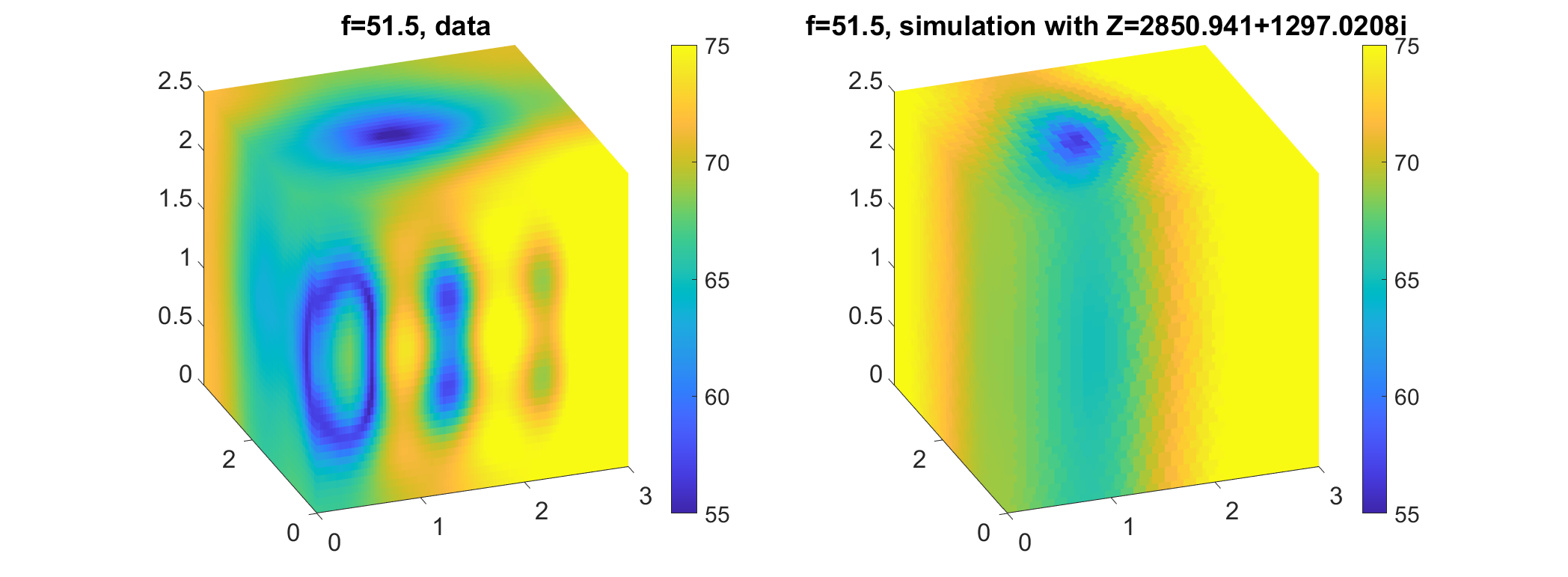}
\caption{The reference solution that produces artificial measurement data (left) and the likelihood maximizer computed by the algorithm (right) in dB (re 1\,Pa$\cdot$s/m$^{3}$) for $f=51.5$\,Hz.}
\label{fig:data1}
\end{figure}
\begin{figure}[htbp]
\centering
\includegraphics[width=1\textwidth]{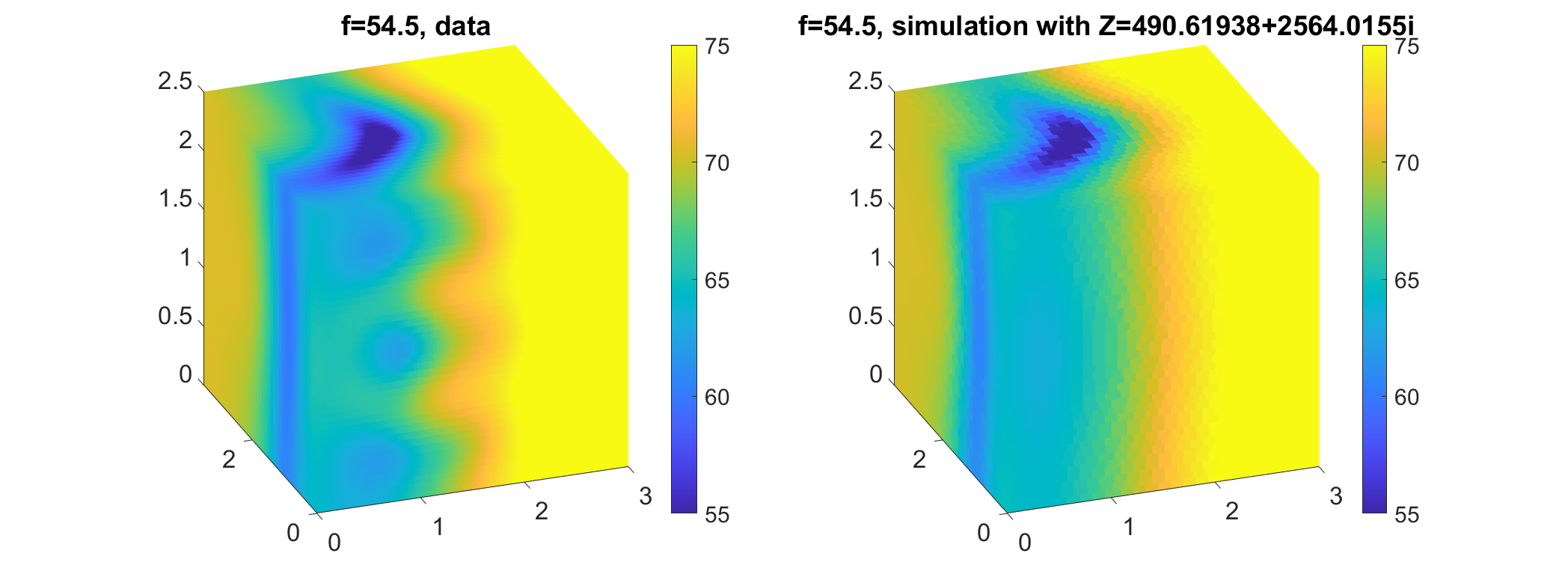}
\caption{The reference solution that produces artificial measurement data (left) and the likelihood maximizer computed by the algorithm (right) in dB (re 1\,Pa$\cdot$s/m$^{3}$) for $f=54.5$\,Hz.}
\label{fig:data2}
\end{figure}
\begin{figure}[htbp]
\centering
\includegraphics[width=1\textwidth]{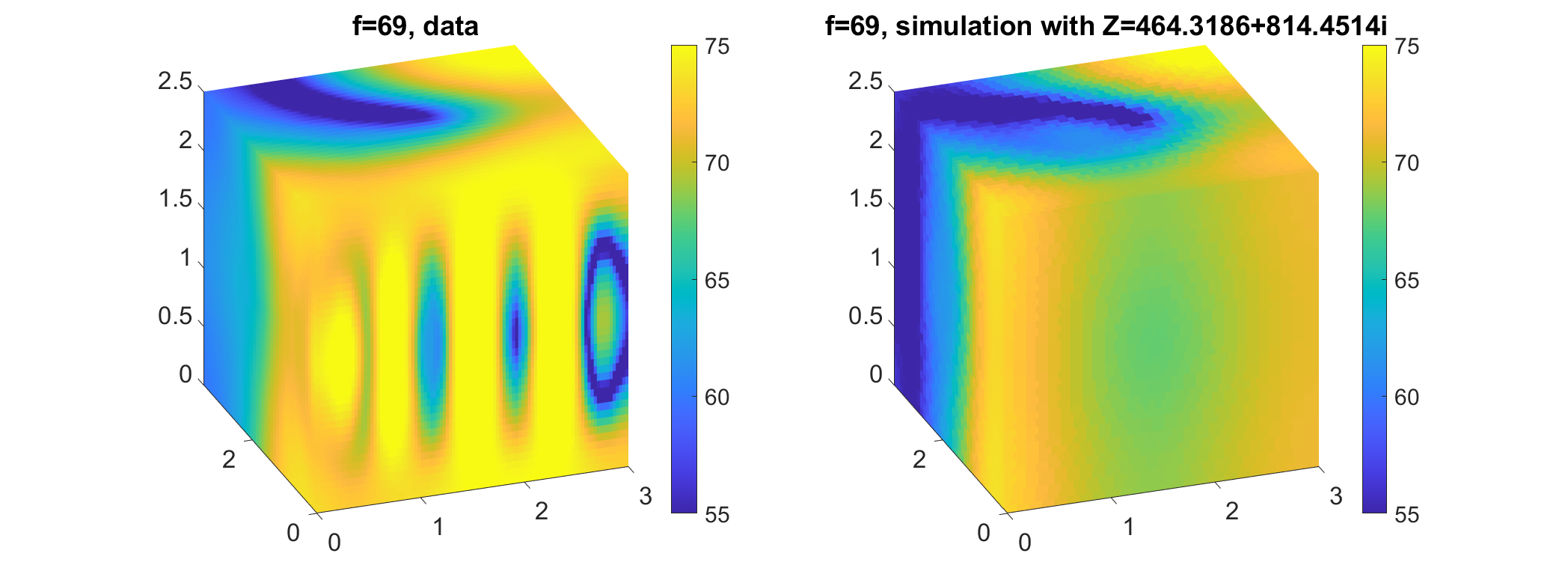}
\caption{The reference solution that produces artificial measurement data (left) and the likelihood maximizer computed by the algorithm (right) in dB (re 1\,Pa$\cdot$s/m$^{3}$) for $f=69$\,Hz.}
\label{fig:data3}
\end{figure}

In contrast to the two typical cases discussed above, the frequency $f=69$\,Hz exhibits an intermediate behaviour, see figure \ref{fig:data3}. The eigenmode of the glass wall is again quite localized near the boundary. The algorithm, however, fails to identify the impedance parameter with the same likelihood as for $f=54.5$\,Hz. In figure \ref{fig:maxLikelihoodCoupledWef} we see, however, that the likelihood drop is not so dramatic as around $f=51.5$\,Hz. As a consequence, we can see clear differences in the solution close to~$\Gamma^{(2)}$, but a more similar qualitative (but not quantitative!) behaviour for parts away from this part of the boundary. 

\section{Outlook}
\label{sec:outlook}
From this point a lot of other extensions are possible. For once, estimating higher moments for the posterior density and computing parametric density functions with those is a way to include properties like skewness in the resulting function.

Another possible extension is to replace the scalar-valued parametric impedance  \linebreak boundary condition by a more general parametric model, e.g., represented by a Karhunen--Lo\'eve expansion, or a non-local boundary condition. In this case a lot more parameters would need to be estimated, depending on the goal.\\

Another thread is the analysis of further intelligent sampling methods like multilevel Monte Carlo or Quasi Monte Carlo (as done in \cite{scheichl2017quasi} for the elliptic problem) or Markov Chain / sequential Monte Carlo method (as done in \cite{engel2019application} for uncertainty in the right hand side of the Helmholtz equation).

\section*{Acknowledgments}
This research was partially funded by the Deutsche For\-schungs\-gemeinschaft (DFG, German
Research Foundation) – Project IDs 352015383 (SFB 1330 C1) and 
444832396 (SPP 2236).


\bibliographystyle{siamplain}
\bibliography{references}

\end{document}